\documentclass[12pt]{article}

\usepackage{latexsym}
\usepackage{indentfirst}
\usepackage{graphicx}
\usepackage{amsmath}
\usepackage{amssymb}
\usepackage{amsthm}
\usepackage{bm}
\usepackage{accents}
\newtheorem{prop}{Proposition}[section]
\newtheorem*{defi*}{Definition}
\newtheorem{defi}[prop]{Definition}

\newtheorem{lem}[prop]{Lemma}
\newtheorem{rem}[prop]{Remark}
\newtheorem{thm}[prop]{Theorem}
\newtheorem{coro}[prop]{Corollary}

\numberwithin{equation}{section}

\begin{document}

\vspace*{-30mm}

\centerline{To appear in {\it Transactions of the American Mathematical Society}}

\vspace{20mm}

\noindent
{\Large\bf  Quasimartingales associated to Markov processes}\\

\noindent
Lucian Beznea\footnote{Simion Stoilow Institute of Mathematics  of the Romanian Academy,
Research unit No. 2, P.O. Box  1-764, RO-014700 Bucharest, Romania,
and University of Bucharest, Faculty of Mathematics and Computer Science
(e-mail: lucian.beznea@imar.ro)} and   
Iulian C\^{i}mpean\footnote{Simion Stoilow Institute of Mathematics  of the Romanian Academy,
Research unit No. 2, P.O. Box  1-764, RO-014700 Bucharest, Romania,
(e-mail: iulian.cimpean@imar.ro)}
\\[5mm]

\noindent
{\bf Abstract.}
For a fixed right process $X$ we investigate those functions $u$ for which $u(X)$ is a quasimartingale. 
We prove that $u(X)$ is a quasimartingale if and only if $u$ is the difference of two finite excessive functions.
In particular, we show that the quasimartingale nature of $u$ is preserved under killing, time change, or Bochner subordination.
The study relies on an analytic reformulation of the quasimartingale property for $u(X)$ in terms of a certain variation of $u$ with respect to the transition function of the process. 
We provide sufficient conditions under which $u(X)$ is a quasimartingale, and finally, we extend to the case of semi-Dirichlet forms a semimartingale characterization of such functionals for symmetric Markov processes, due to Fukushima. \\[2mm]




\section{Introduction}

Let us consider a (right) Markov process 
$X = (\Omega, \mathcal{F}, \mathcal{F}_t, X_t, \mathbb{P}^x)$ with state space $E$. 
In the celebrated paper \cite{CiJaPrSh80}, the authors prove that a real-valued function $u$ on $E$ 
has the property that $u(X)$ is a semimartingale for each $\mathbb{P}^x$ 
if and only if there exists a sequence of finely open sets $(E_n)_{n \geq 1}$ 
such that $\bigcup_n E_n = E$, the exit times $T_n$ of $E_n$ tend to infinity a.s., 
and $u$ is the difference of two $1$-excessive functions on each $E_n$.
This characterization was later approached by Fukushima in \cite{Fu99} from a Dirichlet forms theory perspective.
More precisely, he showed that if $X$ is associated with a symmetric Dirichlet form 
$(\mathcal{E}, \mathcal{F})$ and $u \in \mathcal{F}$, then $\widetilde{u}(X)$ 
is a semimartingale if and only if there exist a nest $(F_n)_{n\geq 1}$ 
and constants $(c_n)_{n\geq 1}$ such that for each $n\geq 1$
 
\begin{equation} \label{lineq}
|\mathcal{E}(u,v)| \leq c_n \|v\|_\infty \; \; \mbox{for all} \; v \in \mathcal{F}_{b, F_n},
\end{equation}
here $\widetilde{u}$ denotes a quasi-continuous version of $u$.
The ideea of Fukushima in order to prove the sufficiency of inequality (\ref{lineq}) 
was to assume first that $\mathcal{E}$ is a regular Dirichlet form so that, 
by Riesz representation, one has $\mathcal{E}(u,v)=\nu(v)$ for some Radon measure $\nu$ on $E$.
The next step was to show that $\nu$ is a smooth measure, which means that the CAF 
from Fukushima decomposition is of bounded variation, hence $\widetilde{u}(X)$ is a semimartingale.
The extension to quasi-regular symmetric Dirichlet forms was achieved via the so called "transfer method".
This result was then used by the author in order to develop a deep stochastic counterpart of BV functions in both finite and infinite dimensions; beside the above mentioned paper, 
we refer the reader also to \cite{Fu00} and the references therein.
As a matter of fact, the approach using Dirichlet forms dates 
back to the work of Bass and Hsu in \cite{BaHs90} 
where they showed that the reflected Brownian motion 
in a Lipschitz domain is a semimartingale, result which was later 
extended to (strong) Caccioppoli sets in \cite{ChFiWi93}, 
where the authors investigate the quasimartingale structure of the process.
It is worth to mention that in \cite{ChFiWi93} 
the authors consider quasimartingales only on finite intervals and not on the entire positive semi-axis, 
as we do (see Definition \ref{defi 2.1}).
Although it might seem a small difference, it is in fact the key point which makes our hole study achievable and, 
to the best of our knowledge, new. 

The aim of this paper is twofold: firstly, we investigate those real-valued functions $u$ on $E$ 
for which $u(X)$ is a quasimartingale, and
secondly, we study those functions $u$ for which $u(X)$ is a semimartingale 
by looking at their local quasimartingale structure.
We briefly present below the structure and the main results of the paper:

In Section 2 we show that the quasimartingale property of $u(X)$ may be reformulated in terms of the variation 
$$
V(u):= \mathop{\sup}\limits_{\tau}\{\mathop{\sum}\limits_{i=1}^{n} P_{t_{i-1}} |u - P_{t_i - t_{i-1}}u| + P_{t_n}|u|\}
$$
of $u$ w.r.t. the semigroup $(P_t)_{t \geq 0}$ of the process, 
which allows us to perform the study from a purely analytic point of view. 
The central results are Theorem \ref{thm 2.6} mainly saying that 
$\{ x \in E \; : \; u(X) \; \mbox{is a quasimartingale w.r.t.} \; \mathbb{P}^x \} = \{V(u)< \infty\}$, 
and Corollary \ref{coro 2.7} according to which $u(X)$ 
is a quasimartingale (which by convention means for all $\mathbb{P}^x, x \in E$) 
if and only if $u$ may be decomposed as the difference of two finite excessive functions.
In particular, if the process is irreducible and 
$(e^{-\alpha t}u(X_t))_{t \geq 0}$ is a $\mathbb{P}^{x_0}$-quasimartingale 
for one $x_0 \in E$, then it is a $\mathbb{P}^{x}$-quasimartingale for all $x \in E$.
A Riesz type decomposition and some remarks on the space of 
differences of excessive functions are discussed in the end of the section.

In Section 3 we show that the quasimartingale property of functions 
is preserved under killing, time change, and Bochner subordination.
In addition, we show that for a multiplicative functional $M$ with permanent points $E_M$, $(e^{-\alpha t}M_t u(X_t))_t$ is a quasimartingale if and only if $(e^{-\alpha t}u|_{E_M}(X^M))_t$ is a quasimartingale, where $X^M$ stands for the killed process by $M$; see Corollary \ref{coro 3.3}.
Also, in Proposition \ref{prop 3.5} we show that if $(e^{-\alpha t}u(X_t))_t$ is a quasimartingale, then so is the process $(e^{-\alpha \tau_t}u(Y_t))_t$, where $\tau$ is the inverse of an additive functional of $X$ and $Y$ denotes the corresponding time change process. 

In Section 4 we provide tractable conditions for $u$ such that $(e^{-\alpha t}u(X_t))_t$ is a quasimartingale.
We distinguish two ways of considering such conditions, which we treat separately: 
the first one involves the resolvent $\mathcal{U}=(U_\alpha)_{\alpha}$ of the process, while the second approach is performed in an $L^p(\mu)$-context, where $\mu$ is a $\sigma$-finite sub-invariant measure.
On brief, the key point is to search for an estimate of the type $U_\alpha(|P_tu - u|) \lesssim t$ 
for the first approach, and of the type $\mu(|P_tu-u|f)\lesssim t \|f\|_\infty$ in the $L^p$-context, 
but we refer the reader to Propositions \ref{prop 4.1} and \ref{prop 4.2} for the precise statemens; see also Proposition \ref{prop 4.5} for a condition in terms of the dual generator on $L^p$-spaces.

In the last section we look at quasimartingale and semimartingale functionals from the Dirichlet form theory point of view.
More precisely, if $(\mathcal{E}, \mathcal{F})$ is a (non-symmetric) Dirichlet form, 
then for an element $u \in \mathcal{F}$, an inequality of the type

\begin{equation} \label{ineq}
|\mathcal{E}(u,v)| \leq c \|v\|_\infty \; \; \mbox{for all}\; v \in \mathcal{F}_b
\end{equation}
\noindent
ensures that $(e^{-\alpha t}\widetilde{u}(X))_t$ is a quasimartingale; see Theorem \ref{thm 5.2}.
As a matter of fact, we show that this is true under a more general situation, 
when $\|v\|_\infty$ in (\ref{ineq}) is replaced by $\|v\|_\infty + \|v\|_{L^2(\mu)}$, cf. Theorem \ref{thm 5.1}.  
Then, in Theorem \ref{thm 5.3} we extend the semimartingale characterization 
due to Fukushima mentioned in the beginning of the introduction, to non-symmetric Dirichlet forms.
Furthermore, in Corollary \ref{coro 5.4} we consider the situation when $u$ 
is not necessarily in $\mathcal{F}$ (e.g. $u \in \mathcal{F}_{\rm loc}$), 
under the additional hypothesis that the form has the local property.
At this point we would like to emphasize that in contrast with previous work, in order to prove the sufficiency of conditions (\ref{lineq}) or (\ref{ineq}) we do not use Fukushima decomposition or Revuz correspondence. 
Instead, we employ heavily the results of the previous sections, and in fact, 
this approach enables us to extend Theorem \ref{thm  5.3} to semi-Dirichlet forms 
without further conditions; we do this in Theorem \ref{thm 5.5}.  

The paper ends with a few remarks concerning situations when it is sufficient to check inequalities 
(\ref{lineq}) or (\ref{ineq}) for $v$ belonging to a proper subspace of $\mathcal{F}$, like cores or special standard cores.

\section{Quasimartingales of Markov processes}
Before considering Markov processes, let us recall some classic facts about 
quasimartingales defined on a general probability space.
 
\begin{defi} \label{defi 2.1}
Let $(\Omega, \mathcal{F}, \mathcal{F}_t, \mathbb{P})$ 
be a filtered probability space satisfying the usual hypotheses. 
An $\mathcal{F}_t$-adapted, right-continuous integrable process $(Z_t)_{t \geq 0}$ is called $\mathbb{P}$-{\rm quasimartingale} if
$$
{Var}^\mathbb{P}(Z):= \mathop{\sup}\limits_{\tau} \mathbb{E} \{ \mathop{\sum}\limits_{i = 1}^{n} |\mathbb{E}[Z_{t_i} - Z_{t_{i-1}}|\mathcal{F}_{t_{i-1}}]| + |Z_{t_n}|\} < \infty,
$$
where the supremum is taken over all partitions $\tau : 0 = t_0 \leq t_1 \leq \ldots \leq t_n < \infty$.
\end{defi}

A classic result is Rao's theorem according to which any 
quasimartingale has a unique decomposition as a sum between 
a local martingale and a predictable process with paths of locally integrable variation. 
In fact, the following characterization inspired our work (see e.g. \cite{Pr05}, page 117):

\begin{thm} \label{thm 2.2} 
(Rao) Let $(\Omega, \mathcal{F}, \mathcal{F}_t, \mathbb{P})$ as in Definition \ref{defi 2.1}. 
A real-valued process is a $\mathbb{P}$-quasimartingale 
if and only if it is the difference of two positive right-continuous $\mathcal{F}_t$-adapted supermartingales.
\end{thm}
\noindent
Conversely, one can show that any semimartingale with bounded jumps is locally a quasimartingale.

Hereinafter we consider a right Markov process $X = (\Omega, \mathcal{F}, \mathcal{F}_t, X_t, \mathbb{P}^x)$ 
with state space a Lusin topological space $E$ endowed with the Borel $\sigma$-algebra $\mathcal{B}$, 
transition function $(P_t)_{t \geq 0}$ and resolvent $\mathcal{U} = (U_{\alpha})_{\alpha > 0}$. 
If $X$ has lifetime $\xi$ and cemetry point $\Delta$, 
we make the convention $u(\Delta) = 0$ for all functions $u: E \to [-\infty, + \infty]$.

The aim of this section is to study those functions 
$u : E \to \mathbb{R}$ for which $u(X)$ is a $\mathbb{P}^x$-quasimartingale for all $x \in E$.

\begin{defi} \label{defi 2.3}
Let $\alpha \geq 0$. A real valued $\mathcal{B}$-measurable function $u$ 
is called $\alpha$-{\rm quasimartingale function} for 
$X$ if $(e^{-\alpha t}u(X_t))_{t \geq 0}$ is a $\mathbb{P}^x$-quasimartingale for all $x \in E$.
When $\alpha = 0$ we shall drop the index from notations.
\end{defi}

\begin{rem} \label{rem 2.4} If $u$ is a quasimartingale function then,
$\mathop{\sup}\limits_{t} P_t|u|(x) = \mathop{\sup}\limits_{t}\mathbb{E}^x|u(X_t)| \leq {Var}^{\mathbb{P}^x}(u(X)) < \infty, \; x \in E.$
Also, by the $\mathbb{P}^x$-a.s. right continuity of the trajectories $t \mapsto u(X_t)$, $u$ must be finely continuous; see \cite{BlGe68}, Theorem 4.8.
\end{rem}

\noindent
{\bf Notations.} For a real valued function $u$ and a partition $\tau$ of $\mathbb{R}^+$, $\tau : 0 = t_0 \leq t_1 \leq \ldots \leq t_n < \infty$, we set

\qquad \qquad \qquad \qquad $V_{\tau}^{(P_t)}(u) := \mathop{\sum}\limits_{i=1}^{n} P_{t_{i-1}} |u - P_{t_i - t_{i-1}}u| + P_{t_n}|u|$,

\qquad \qquad \qquad \qquad $V^{(P_t)}(u) := \mathop{\sup}\limits_{\tau}V_{\tau}^{(P_t)}(u)$,

\noindent
where the supremum is taken over all finite partitions of $\mathbb{R}_+$. 
If there is no risk of confusion we shall write $V_{\tau}(u)$ and 
$V(u)$ instead of $V_{\tau}^{(P_t)}(u)$ resp. $V^{(P_t)}(u)$. 
Also, for $\alpha > 0$ we set 
$V_{\tau}^{\alpha}(u) := V_{\tau}^{(P_t^{\alpha})}(u)$ and 
$V^{\alpha}(u) := V^{(P_t^{\alpha})}(u)$, where $P_t^\alpha :=e^{-\alpha t}P_t, \alpha > 0$.

Recall that for $\alpha \geq 0$, a $\mathcal{B}$-measurable function 
$f:E \rightarrow [0, \infty]$ is called {\it $\alpha$-supermedian} if $P_t^\alpha f \leq f$, $t \geq 0$. 
If $f$ is $\alpha$-supermedian and $\lim\limits_{t \to 0} P_t^\alpha f = f$ then it is called {\it $\alpha$-excessive}.
The convex cone of all $\alpha$-supermedian (resp. $\alpha$-excessive) functions is denoted by 
$S(\mathcal{U}^\alpha)$ (resp. $E(\mathcal{U}^\alpha)$).
If $\alpha = 0$ we shall drop the index $\alpha$ from notations.

A set $A \in \mathcal{B}$ is called {\it absorbing} if $R_\alpha^{E\setminus A} 1 = 0$ on $A$, 
where $R_\alpha^A 1 := inf \{ s \in E(\mathcal{U}^\alpha) : s \geq 1_A \}$.
We recall that if $A$ is absorbing then it is finely open, $U_\alpha 1_{E\setminus A} = 0$ on $A$, 
and the restriction of $X$ to $A$ is again a right process; see e.g. \cite{Sh88} or \cite{BeRo11}. 
Standard examples of absorbing sets are $[v=0]$ and $[v<\infty]$ 
for some $v\in E(\mathcal{U}^\alpha)$ and $\alpha \geq 0$. 

\begin{defi} \label{defi 2.5} 
A sequence $(\tau_n)_{n \geq 1}$ of finite partitions of $\mathbb{R}_+$ is called 
{\rm admissible} if it is increasing, $\mathop{\bigcup}\limits_{k \geq 1}\tau_k$ is dense in $\mathbb{R}_+$, 
and if $r \in \mathop{\bigcup}\limits_{k \geq 1}\tau_k$ then $r + \tau_n \subset \mathop{\bigcup}\limits_{k \geq 1}\tau_k$ for all $n \geq 1$.
\end{defi}

The next theorem and its first corollary are the main results of this section.

\begin{thm} \label{thm 2.6}
Let $u$ be a real valued $\mathcal{B}$-measurable function such that $P_t|u| < \infty$ for all $t$. 
Then the following assertions hold.

\vspace{0.2cm}
i) ${Var}^{\mathbb{P}^x}(u(X)) = V(u)(x), \; x \in E$.

\vspace{0.2cm}
ii) If $u_1, u_2 \in S(\mathcal{U})$ s.t. $u = u_1 - u_2$ 
on the set $[u_1 + u_2 < \infty]$ then $V(u) \leq u_1 + u_2$ on $[u_1+ u_2 < \infty]$.

\vspace{0.2cm}
iii) If $u$ is finely continuous, then there exist $u_1, u_2 \in E(\mathcal{U})$ such that 
$[V(u) < \infty] = [u_1 + u_2 < \infty]$ and $u = u_1 - u_2$ on $[V(u) < \infty]$. 
In this case, the set $[V(u) < \infty]$ is absorbing and 
$[V(u) < \infty] = [\mathop{\sup}\limits_{n}V_{\tau_n}(u) < \infty] = [\lim\limits_n V_{\tau_n}(u) < \infty]$ 
for any admissible sequence of partitions $(\tau_n)_n$.
\end{thm}

One of the fundamental connections between potential theory and Markov processes is the relation between excessive functions and (right-continuous) supermartingales. 
More precisely, it is well known that for a non-negative real-valued measurable function $u$ we have that $u(X)$ is a supermartingale if and only if $u$ is excessive; see e.g. \cite{LG06}, Proposition 13.7.1 and Theorem 14.7.1. The following essential consequence of Theorem \ref{thm 2.6} (and its proof), states that this connection may be extended between the space of differences of excessive function on the one hand, and quasimartingales on the other hand, in the same time revealing a Hahn-Jordan type decomposition. 

\begin{coro} \label{coro 2.7} 
A real valued $\mathcal{B}$-measurable function $u$ is a quasimartingale function for $X$ 
if and only if there exist two real-valued functions $u_1, u_2 \in E(\mathcal{U})$ such that $u = u_1 - u_2$; in this case one can take $u_1:=\mathop{\sup}\limits_{n}V_{\tau_n}(u)$, where $(\tau_n)_{n \geq 1}$ is any fixed sequence of admissible partitions of $\mathbb{R}_+$.
\end{coro}

For the proof of Theorem \ref{thm 2.6} we need the following lemma. 
Because we found this result only as an exercise (left for the reader) 
in \cite{Sh88}, Exercise 10.24 or \cite{BlGe68}, Exercise 4.14, we include its complete proof below.
 
The first hitting time of a set $A \in \mathcal{B}$ by the process $X$ is defined by 
$T_A := \inf \{ t > 0 : X_t \in A \}$. It is well known that $T_A$ is a stopping time; 
see \cite{BlGe68} or \cite{Sh88}.

\begin{lem} \label{lem 2.8} 
If $u$ is finely continuous and bounded then so is $P_s u$ for all $s \geq 0$.
\end{lem}

\begin{proof}
 Since $u$ is finely continuous, by \cite{BlGe68}, Theorem 4.8, 
 it follows that the mappings $t \mapsto u(X_t)$ are right continuous a.s. 
Let $s > 0$ and set $f := P_s u$.
In order to show that $f$ is finely continuous it is sufficient to prove that if 
$\varepsilon > 0$ then $x$ is irregular for 
$A = f^{-1}([f(x) + \varepsilon, \infty))$ and $B = f^{-1}((-\infty, f(x) - \varepsilon])$. 
We treat only the first case. Let $(A_n)_n$ be an increasing sequence of closed sets such that 
$T_{A_n} \searrow T_A \; \mathbb{P}^x$-a.s. 
By the zero-one law (\cite{BlGe68}, Proposition 5.17), $\mathbb{P}^x(T_A = 0) \in \{0,1\}$. 
Assume that $x$ is regular for $A$, i.e. $T_A = 0$ $\mathbb{P}^x$-a.s. 
Then by the strong Markov property and dominated convergence theorem, 
$\mathbb{E}^xf(X_{T_{A_n}}) = 
\mathbb{E}^x\{\mathbb{E}^x[u(X_{s + T_{A_n}})| \mathcal{F}_{T_{A_n}}]\} = 
\mathbb{E}^x u(X_{s + T_{A_n}}) \mathop{\longrightarrow}\limits_{n} f(x)$. 
On the other hand, by the definition of $T_{A_n}$ we have that $f(X_{T_{A_n}}) \geq f(x) + \varepsilon$, 
which contradicts the previous convergence.
\end{proof}

\vspace{0.2cm}

\begin{proof}[Proof of Theorem \ref{thm 2.6}.] 

i). By the Markov property, for all  $x \in E$

${Var}^{\mathbb{P}^x}(u(X)) = 
\mathop{\sup}\limits_{\tau} \mathbb{E}^x\{ \mathop{\sum}\limits_{i=1}^{n} |u(X_{t_{i-1}}) - 
\mathbb{E}^x[u(X_{t_i})|\mathcal{F}_{t_{i-1}}]| + |u(X_{t_n})| \} $

\vspace{0.2cm}

\qquad \qquad \qquad 
$= \mathop{\sup}\limits_{\tau} \{ \mathbb{E}^x\{\mathop{\sum}\limits_{i=1}^{n} |u(X_{t_{i-1}}) - P_{t_i - t_{i-1}} u(X_{t_{i-1}})| + 
P_{t_n}|u|(x)\} \}$

\vspace{0.2cm}

\qquad \qquad \qquad $
= \mathop{\sup}\limits_{\tau} \{ \mathop{\sum}\limits_{i=1}^{n} \mathbb{E}^x[|u - P_{t_i - t_{i-1}}u|(X_{t_{i-1}})] + P_{t_n}|u|(x) \}$

\vspace{0.2cm}

\qquad \qquad \qquad 
$= \mathop{\sup}\limits_{\tau} \{ \mathop{\sum}\limits_{i=1}^{n} P_{t_{i-1}}|u - P_{t_i - t_{i-1}}u|(x) + P_{t_n}|u|(x) \} = V(u)(x)$.

\vspace{0.2cm}

\noindent 
Note that the above expressions make sense because by hypothesis, $P_t|u| < \infty$ for all $t$.

ii). Since $u_1, u_2 \in S(\mathcal{U})$ we have that $A:= [u_1 + u_2 < \infty]$ 
satisfies $P_t1_{A^c} = \mathop{\lim}\limits_{n} P_t1_{[u_1 +u_2 > n]} 
\leq \mathop{\lim}\limits_{n} \dfrac{1}{n} P_t(1_{[u_1 +u_2 > n]} (u_1 +u_2)) 
\leq \mathop{\lim}\limits_{n} \dfrac{u_1 +u_2}{n} = 0$ on $A$ for all $t > 0$. 
This leads to $1_AP_t f = 1_A P_t(f 1_A)$ for all $\mathcal{B}$-measurable $f$ for which $P_t|f| < \infty$. 
Indeed, $|1_A P_t(f 1_{A^c})| \leq 1_A P_t(|f|1_{A^c}) = 
\mathop{\sup}\limits_{n}1_A P_t((|f|\wedge n)1_{A^c}) \leq \mathop{\sup}\limits_{n} 1_A n P_t 1_{A^c} = 0$.

By the previous remarks, we get

\vspace{0.2cm}

\qquad $\; \; 1_A V(u) = \mathop{\sup}\limits_{\tau}\{ \mathop{\sum}\limits_{i=1}^{n} 1_A P_{t_{i-1}}|u - P_{t_i - t_{i-1}}u| + 1_A P_{t_n}|u| \}$ 

\vspace{0.2cm}
\qquad \qquad \qquad $ = \mathop{\sup}\limits_{\tau} \{ \mathop{\sum}\limits_{i=1}^{n} 1_A P_{t_{i-1}}|1_A u - 1_A P_{t_i - t_{i-1}}(1_A u)| + 1_A P_{t_n}1_A|u| \}$

\vspace{0.2cm}
\qquad \qquad \qquad $\leq \mathop{\sup}\limits_{\tau} \{ \mathop{\sum}\limits_{i=1}^{n} 1_A P_{t_{i-1}} |1_A u_1 - 1_A P_{t_i - t_{i-1}}(1_A u_1)| + 1_A P_{t_n}(1_A u_1) \}$

\vspace{0.2cm}
\qquad \qquad \qquad \; \; $+ \mathop{\sup}\limits_{\tau} \{ \mathop{\sum}\limits_{i=1}^{n} 1_A P_{t_{i-1}} |1_A u_2 - 1_A P_{t_i - t_{i-1}}(1_A u_2)| + 1_A P_{t_n}(1_A u_2) \}$

\vspace{0.2cm}
\qquad \qquad \qquad $= 1_A \mathop{\sup}\limits_{\tau} \{ \mathop{\sum}\limits_{i=1}^{n} [P_{t_{i-1}} u_1 - P_{t_i} u_1] + P_{t_n} u_1 \} + 1_A \mathop{\sup}\limits_{\tau} \{ \mathop{\sum}\limits_{i=1}^{n} [P_{t_{i-1}} u_2 - P_{t_i}u_2] + P_{t_n} u_2 \}$

\vspace{0.2cm}
\qquad \qquad \qquad $= 1_A(u_1 + u_2)$.

\vspace{0.2cm}
iii). For each partition $\tau : 0 = t_0 \leq t_1 \leq \ldots \leq t_n < \infty$, we set
$$
u^{\tau}_1 := \mathop{\sum}\limits_{i=1}^{n} P_{t_{i-1}}(u - P_{t_i - t_{i-1}} u)^+ + P_{t_n}(u^+) < \infty
$$
$$
u^{\tau}_2 := \mathop{\sum}\limits_{i=1}^{n} P_{t_{i-1}}(u - P_{t_i - t_{i-1}} u)^- + P_{t_n}(u^-) < \infty.
$$

Let $\prec$ denote the ordering of set containment and suppose that 
$\sigma$ and $\tau$ are two finite partitions of $\mathbb{R}_+$ s.t. $\sigma \prec \tau$. 
We claim that $u^{\sigma}_i \leq u^{\tau}_i$, $i = \overline{1,2}$. 
To see this, let $\sigma : 0 = t_0 \leq t_1 \leq \ldots \leq t_n < \infty$ 
and note that it is enough to consider $\tau$ as a partition obtained from $\sigma$ 
by adding an extra point $t$ before $t_1$, after $t_n$, 
or between some $t_i$ and $t_{i+1}$. 
In the first case we have 
$(u - P_{t_1}u)^{\pm} \leq (u - P_t u)^{\pm} + (P_t(u - P_{t_1 - t}u))^{\pm} \leq (u - P_t u)^{\pm} + P_t(u - P_{t_1 - t}u)^{\pm}$.

\noindent
If $t \geq t_n$ then $P_{t_n}(u^{\pm}) \leq P_{t_n}(u - P_{t - t_n}u)^{\pm} + P_t(u^{\pm})$, and if $t_i \leq t \leq t_{i+1}$, then $(u - P_{t_{i+1} - t_i}u)^{\pm} \leq (P_{t - t_i}(u - P_{t_{i+1} - t}u))^{\pm} + (u - P_{t - t_i}u)^{\pm} \leq P_{t - t_i}(u - P_{t_{i+1} - t}u)^{\pm} + (u - P_{t - t_i}u)^{\pm}$, hence $P_{t_i}(u - P_{t_{i+1} - t_i}u)^{\pm} \leq P_t(u - P_{t_{i+1} - t}u)^{\pm} + P_{t_i}(u - P_{t - t_i}u)^{\pm}$.

\noindent
Therefore, $u^{\sigma}_i \leq u^{\tau}_i$, $i = \overline{1,2}$.

Let now $(\tau_n)_{n \geq 1}$ be an admissible sequence of partitions of $\mathbb{R}_+$ and define
$$
u_1 := \mathop{\sup}\limits_{n} u^{\tau_n}_1 = \mathop{\lim}\limits_{n} u^{\tau_n}_1 \; {\rm and} \;
u_2 := \mathop{\sup}\limits_{n} u^{\tau_n}_2 = \mathop{\lim}\limits_{n} u^{\tau_n}_2.
$$

\noindent
Then $u_1 + u_2 = \mathop{\sup}\limits_{n} V_{\tau_n}(u) < \infty$ on $[V(u) < \infty]$.

Now, if $r \in \mathop{\bigcup}\limits_{n \geq 1} \tau_n$,

\quad $P_r u_1 = \mathop{\sup}\limits_{n} P_r u^{\tau_n}_1 = \mathop{\sup}\limits_{n} \{ \mathop{\sum}\limits_{i=1}^{n} P_{r + t_{i-1}}(u - P_{t_i - t_{i-1}}u)^+ + P_{r + t_n}(u^+) \}$

\vspace{0.2cm}
\quad \qquad \;$= \mathop{\sup}\limits_{n} \{\mathop{\sum}\limits_{i=1}^{n} P_{r + t_{i-1}}(u - P_{r + t_i - (r + t_{i-1})} u)^+ + P_{r + t_n}(u^+) \}$

\vspace{0.2cm}
\quad \qquad \;$\leq u_1$,

\vspace{0.2cm}
\noindent
because the last supremum is taken over a class of partitions included in $\{\tau_n : n \geq 1\}$.
Analogously, $P_r u_2 \leq u_2$ for all $r \in \mathop{\bigcup}\limits_{n \geq 1}\tau_n$.
Then, 

$u_2 + u = \mathop{\sup}\limits_{n} \left\{ \mathop{\sum}\limits_{i=1}^{n} P_{t_{i-1}}(u - P_{t_i - t_{i-1}}u)^- + P_{t_n}(u^-)  + \mathop{\sum}\limits_{i=1}^{n} P_{t_{i-1}}(u - P_{t_i - t_{i-1}}u) + P_{t_n}u \right\}$

\vspace{0.2cm}
\qquad \quad $= \mathop{\sup}\limits_{n} \left\{ \mathop{\sum}\limits_{i=1}^{n} P_{t_{i-1}}(u - P_{t_i - t_{i-1}}u)^- + P_{t_n}(u^-) \right\}$

\vspace{0.2cm}
\qquad \quad $= u_1$,

\vspace{0.2cm}
\noindent and $u = u_1 - u_2$ on $[\mathop{\sup}\limits_{n} V_{\tau_n}(n) < \infty]$.

\vspace{0.2cm}

{\it Case 1}. Assume that $u$ is lower bounded. 
We claim that $u_1,u_2 \in E(\mathcal{U})$ and $[V(u) < \infty] = [\mathop{\sup}\limits_{n} V_{\tau_n}(u) < \infty]$.
First, note that if $u_1, u_2 \in E(\mathcal{U})$, 
since $u_1 + u_2 = \mathop{\sup}\limits_{n} V_{\tau_n}(u)$, 
$[V(u) < \infty] \subset [\mathop{\sup}\limits_{n} V_{\tau_n}(u) < \infty]$, and $u = u_1 - u_2$ on $[u_1 + u_2 < \infty]$, 
by ii) we obtain $[V(u) < \infty] = [\mathop{\sup}\limits_{n} V_{\tau_n}(u) < \infty]$.

It remains to show that $u_1, u_2 \in E(\mathcal{U})$. 
By Lemma \ref{lem 2.8}, the functions $\varphi_{k,l}^n 
:= \mathop{\sum}\limits_{i=1}^{n} P_{t_{i-1}} [(u - P_{t_i - t_{i-1}}(u \wedge k))^- \wedge l]$ 
are finely continuous and 
$u_2 = \mathop{\sup}\limits_{n}\mathop{\sup}\limits_{k}\mathop{\sup}\limits_{l} \varphi_{k,l}^n$ 
is finely lower semi-continuous. 
Moreover, if $t \in \mathbb{R}_+$ and 
$(t_j)_j \subset \mathop{\bigcup}\limits_{n \geq 1} \tau_n$, $t_j \searrow t$, then

$$
P_t u_2 = \mathop{\sup}\limits_{n}\mathop{\sup}\limits_{k}\mathop{\sup}\limits_{l} P_t \varphi_{k,l}^n = \mathop{\sup}\limits_{n}\mathop{\sup}\limits_{k}\mathop{\sup}\limits_{l}\mathop{\lim}\limits_{j} P_{t_j} \varphi_{k,l}^n \leq \mathop{\lim\inf}\limits_{j} P_{t_j} u_2 \leq u_2,
$$
so $u_2$ is supermedian, and by \cite{BeBo04}, Corollary 1.3.4, it is excessive.
Now, since $u_1 = u_2 + u$ is finely continuous for $t \in \mathbb{R}_+$ and $(t_j)_j$ as before,
$$
P_t u_1 = \mathop{\sup}\limits_{k} P_t(u_1 \wedge k) = \mathop{\sup}\limits_{k} \mathop{\lim}\limits_{j} P_{t_j}(u_1 \wedge k) \leq u_1,
$$

\noindent and $u_1 \in E(\mathcal{U})$.

\vspace{0.2cm}

{\it Case 2.} Let now $u$ be arbitrary. 
Then $u^+ = u_1 - u_1 \wedge u_2$ and $u^- = u_2 - u_1 \wedge u_2$ 
are finely continuous and of course, lower bounded.
Applying Case 1 to $u^+$ and $u^-$ we have that $u = u^+ - u^-$ 
is the difference of two real-valued excessive functions on $[V(u^+) < \infty] \cap [V(u^-) < \infty]$.
Let us show that $[V(u) < \infty] = [\mathop{\sup}\limits_{n} V_{\tau_n}(u) < \infty] = [V(u^+) < 
\infty] \cap [V(u^-) < \infty]$, which completes the proof. 
Arguing as in the proof of ii), one can check that 
$A = [u_1 + u_2 < \infty]=[\mathop{\sup}\limits_{n} V_{\tau_n}(u) < \infty]$ 
satisfies $P_r 1_{A^c} = 0$ on $A$ for all $r \in \mathop{\bigcup}\limits_{n \geq 1} \tau_n$, 
and further, $V(u^{\pm}) = \mathop{\sup}\limits_{n} V_{\tau_n}(u^{\pm}) \leq u_1 + u_2$ on $A$. 
Taking into account the sub-additivity of $f \mapsto V(f)$,
$$
[V(u) < \infty] \subset [\mathop{\sup}\limits_{n} V_{\tau_n}(u) < \infty] = A \subset [V(u^+) + V(u^-) < \infty] \subset [V(u) < \infty].
$$
\end{proof}

We say that the process $X$ is {\it irreducible} (in the strong sense) if the only non-empty absorbing set is the hole space $E$.
Often in practice, the irreducibility of $\mathcal{U}$ is ensured by the {\it strong Feller} properly 
(i.e. $U_\alpha$ maps bounded measurable functions into continuous ones) in association with the 
{\it topological irreducibility} (i.e. $U_\alpha 1_D > 0$ for all open sets $D \subset E$); cf. e.g. \cite{Ha10}.

\begin{coro} \label{coro 2.9} 
Let $u$ be a real-valued $\mathcal{B}$-measurable finely continuous function and assume that there exists $x_0 \in E$ such that $(e^{-\alpha t}u(X_t))_{t\geq 0}$ is a $\mathbb{P}^{x_0}$-quasimartingale for some $\alpha \geq 0$.
The following assertions hold.

\vspace{0.2cm}

i) If $\mathcal{U}$ is irreducible then $(e^{-\alpha t}u(X_t))_{t \geq 0}$ is a $\mathbb{P}^x$-quasimartingale for all $x \in E$.

\vspace{0.2cm}

ii) If $\mathcal{U}$ is strong Feller and topologically irreducible then $\mathcal{U}$ is irreducible.
\end{coro}

\begin{proof}
i). By Proposition \ref{prop 3.1} below we have that 
$V^{\alpha}(u)(x_0) = Var^{\mathbb{P}^{x_0}}((e^{-\alpha t}u(X_t))_{t \geq 0}) < \infty$, 
hence $A:=[V^{\alpha}(u) < \infty]$ is absorbing (cf. Theorem \ref{thm 2.6}, iii)) and non-empty.
Since $\mathcal{U}$ is irreducible it follows that $A=E$.

\vspace{0.2cm}
ii). Let $B \in \mathcal{B}$ be absorbing and set $E_0 := [U_1 1_{E\setminus B} = 0] \supset B$.
The strong Feller property implies that $E\setminus E_0$ is an open set 
and $1_{E \setminus B} \geq U_1 1_{E \setminus B} \geq U_1 1_{E \setminus E_0}$ leads to $E_0 = E$.
\end{proof}

Following \cite{Ge80}, $X$ is called {\it recurrent} if either $U1_B=0$ or $U1_B=\infty$ for all $B \in \mathcal{B}$.
Getoor showed that $\mathcal{U}$ is recurrent if and only if any excessive function is constant, 
hence Corollary \ref{coro 2.7} gives the following quasimartingale characterization of recurrence. 

\begin{coro} \label{coro 2.10}
$X$ is recurrent if and only if every quasimartingale function is constant.
\end{coro}

\noindent
{\bf A Riesz type decomposition.}
Extending \cite{Me66} (see also \cite{Sh88}, Chapter VI), a quasimartingale function $f$ 
is called {\it (locally) harmonic} if $f(X)$ is a $\mathbb{P}^x$-(local) martingale for all $x \in E$; 
it is called a {\it potential function of class $(D)$} if for any sequence of stopping times 
$(T_n)_{n} \nearrow \infty$, $\mathbb{E}^x[f(X_{T_n})] \to 0$.

\begin{thm} \label{thm 2.11}
If $u$ is a quasimartingale function for $X$, then $u$ may be decomposed as 
$u = h + v$, where $h$ is locally harmonic and $v$ is a potential function of class $(D)$.
\end{thm} 

\begin{proof} 
It follows by Corollary \ref{coro 2.7} and \cite{Sh88}, Theorem (51.10).
\end{proof}

\vspace{0.3cm}

\noindent
{\bf The space of differences of excessive functions.}
We saw that the space of $\alpha$-quasimartingale functions of $X$ 
is in fact  the space of differences of real-valued $\alpha$-excessive functions. 
We end this section by collecting some useful observations on the dependence on 
$\alpha$ of the above mentioned spaces, in the same spirit as \cite{BeLu16}, Remark 2.1.

Recall that $\mathcal{U}$ is called $m$-transient 
($m$ is a fixed $\sigma$-finite sub-invariant measure for $\mathcal{U}$) 
if there exists $0 < f \in L^1(m)$ such that $Uf < \infty$ $m$-a.e.

\begin{prop} \label{prop 2.12} The following assertions hold.

\vspace{0.2cm}
i) For $\alpha, \beta \geq 0$, if $v \in E(\mathcal{U}_{\alpha})$ is real-valued such that $U_{\beta}v < \infty$, 
then $v$ is a difference of two real-valued $\beta$-excessive functions.
In particular, $bE(\mathcal{U_{\alpha}}) - bE(\mathcal{U}_{\alpha})$ is independent of $\alpha > 0$.

ii) Let $m$ be a $\sigma$-finite sub-invariant measure for $\mathcal{U}$.
Then:

\vspace{0.2cm}
\quad ii.1) If $\alpha, \beta \geq 0$, $v \in E(\mathcal{U}_{\alpha})$ and $U_{\beta} v < \infty$ $m$-a.e. then $v$ is $m$-a.e. (hence q.e.) the difference of two $\beta$-excessive functions.
In particular, the $L^p$-subspaces $L^p(m) \cap E(\mathcal{U}_{\alpha}) - L^p(m) \cap E(\mathcal{U}_{\alpha})$ are independent of $\alpha > 0$ for all $1 \leq p \leq \infty$.

\quad ii.2) If $\mathcal{U}$ is $m$-transient, then the $L^1$-subspaces $L^1(m) \cap E(\mathcal{U}_{\alpha}) - L^1(m) \cap E(\mathcal{U}_{\alpha})$ are independent of $\alpha \geq 0$.
\end{prop}

\begin{proof}
i). Of course, we need to consider only the case $\beta < \alpha$.
Let $w := v + (\alpha - \beta)U_{\beta}v$. 
Then by hypothesis, $w < \infty$ and it is straightforward to check that $w$ is $\beta$-excessive.
Hence $v = w - (\alpha - \beta)U_{\beta}v \in E(\mathcal{U}_{\beta}) - E(\mathcal{U}_{\beta})$.

The proof of ii.1) is similar to the one for assertion i).

ii.2). By ii.1), it is sufficient to show that if $0 \leq v \in L^1(m)$ then $Uv < \infty$ $m$-a.e. 
But this is true by a characterization of $m$-transience; see \cite{BeCiRo15}.
\end{proof}

\section{Quasimartingale functions of transformed Markov processes}

As in Section 2, $X = (\Omega, \mathcal{F}, \mathcal{F}_t, X_t, \mathbb{P}^x)$ is a right Markov process on $E$. 
Before we move on, we would like to remark that although in Section 1 
we considered only $\mathcal{B}$-measurable functions, 
the results obtained there remain valid for functions measurable with respect to $\mathcal{B}^u$, 
the $\sigma$-algebra of all universally measurable sets in $E$.

\vspace{0.3cm}

\noindent
{\bf Quasimartingales under killing.} 
Let $M:=(M_t)_{t \geq 0}$  be a right continuous decreasing multiplicative functionals $(MF)$ of $X$ 
and $E_M$ be the set of permanent points for $M$, $E_M := \{ x \in E : \mathbb{P}^x(M_0 = 1) = 1 \}$.
As in \cite{Sh88}, Proposition 56.5, define the kernels on $p\mathcal{B}^u$ 
by setting for $f \in p\mathcal{B}^u$, $\alpha \geq 0$, and $t \geq 0$
$$
P_M^{\alpha} f(x) := \left\{ \begin{array}{l}
\mathbb{E}^x\int_{0}^{\infty} e^{-\alpha t} f(X_t)dM_t, \; x \in E_M\\[3mm] 
f(x), \quad \quad x \in E \setminus E_M,
\end{array}\right.
$$

\qquad \qquad \qquad \quad \; $Q_t f(x) := \mathbb{E}^x\{f(X_t)M_t\}$,

\vspace{0.2cm}
\qquad \qquad \qquad \quad \; $W_{\alpha}f(x) := \mathbb{E}^x\int_{0}^{\infty}e^{-\alpha t}M_t f(X_t)dt$.

\vspace{0.2cm}
It is well known that $(Q_t)_{t}$ is a sub-Markovian semigroup of kernels on 
$(E, \mathcal{B}^u)$ whose resolvent is $\mathcal{W} = (W_{\alpha})_{\alpha} \geq 0$.

\begin{prop} \label{prop 3.1} Let $u$ be a real-valued $\mathcal{B}^u$-measurable function such that 
$P_t|u| < \infty$ for all $t \geq 0$. Then for all $x \in E$,
$$
{Var}^{\mathbb{P}^x}(Mu(X)) = V^{(Q_t)}u(x).
$$
\end{prop}

\begin{proof} 
For $x \in E$,
$$
{Var}^{\mathbb{P}^x}(Mu(X)) = \mathop{\sup}\limits_{\tau} \mathbb{E}^x \left\{ \mathop{\sum}\limits_{i=1}^{n} |\mathbb{E}^x[M_{t_{i-1}} u(X_{t_{i-1}}) - M_{t_i}u(X_{t_i})|\mathcal{F}_{t_{i-1}}]| + M_{t_n}|u|(X_{t_n}) \right\}
$$

\vspace{0.2cm}
\qquad \qquad \qquad \; $= \mathop{\sup}\limits_{\tau} \mathbb{E}^x\left\{ \mathop{\sum}\limits_{i=1}^{n} |M_{t_{i-1}} u(X_{t_{i-1}}) - M_{t_{i-1}}Q_{t_i - t_{i-1}} u(X_{t_{i-1}})| + M_{t_n}|u|(X_{t_n}) \right\}$

\vspace{0.2cm}
\qquad \qquad \qquad \; $= \mathop{\sup}\limits_{\tau} \left\{ \mathop{\sum}\limits_{i=1}^{n} \mathbb{E}^x[M_{t_{i-1}} | u - Q_{t_i - t_{i-1}}u| (X_{t_{i-1}})] + Q_{t_n}|u|(x) \right\}$

\vspace{0.2cm}
\qquad \qquad \qquad \; $= \mathop{\sup}\limits_{\tau} \left\{ \mathop{\sum}\limits_{i=1}^{n} Q_{t_{i-1}} |u - Q_{t_i - t_{i-1}} u|(x) + Q_{t_n}|u|(x) \right\} $

\vspace{0.2cm}
\qquad \qquad \qquad \; $= V^{(Q_t)}u(x)$.
\end{proof}

\begin{coro} \label{coro 3.2} 
Let $u$ be a real-valued $\mathcal{B}^u$-measurable function. 
If $\alpha \geq 0$, then $u$ is an $\alpha$-quasimartingale function 
if and only if it is the difference of two real-valued $\alpha$-excessive functions.
\end{coro}

\vspace{0.3cm}

If $M$ is exact, then $E_M$ is finely open and the restriction 
$Q_t|_{E_M}$ of $(Q_t)_{t \geq 0}$ to $E_M$ is the transition function of a right Markov process $(X_t^M)_{t \geq 0}$ on $E_M$; see \cite{Sh88}, Chapter VII.

\vspace{0.3cm}

\begin{coro} \label{coro 3.3}
Assume that $M$ is perfect. 
Then the following assertions hold.

i) Let $f$ be a real-valued $\mathcal{B}^u$-measurable function such that 
$U_{\alpha}|f| < \infty$ for some $\alpha \geq 0$ and set $u := W_{\alpha}f$.
Then $u$ is an $\alpha$-quasimartingale function for $X$.

ii) Let $u$ be a real-valued $\mathcal{B}^u$-measurable function, 
such that $Q_t|u| < \infty$ for all $t \geq 0$. 
Then for all $\alpha \geq 0$
$$
V^{(Q_t^{\alpha})}(u) = \left\{ \begin{array}{l}
V^{(Q_t^{\alpha}|_{E_M})}(u|_{E_M})(x), \; x \in E_M\\[3mm] 
0, \quad \quad x \in E \setminus E_M.
\end{array}\right.
$$
In particular, if $u$ is finely continuous then for all $\alpha \geq 0$, 
$(e^{-\alpha t}M_t u(X_t))_t$ is a $\mathbb{P}^x$-quasimartingale for all $x \in E$ 
if and only if $u|_{E_M}$ is an $\alpha$-quasimartingale function for $X^M$.
\end{coro}

\begin{proof}
 i). Clearly, it is enough to consider $f \geq 0$. Then, the assertion follows since $u = U_{\alpha}f - P_M^{\alpha}U_\alpha f$ and $P_M^{\alpha}U_{\alpha}f \in E(\mathcal{U}^{\alpha})$; see e.g. \cite{Sh88}, Proposition 56.5.

ii). The first assertion follows easily since $Q_t f \equiv 0$ on $E_M$ and $M_t \equiv 0$ $\mathbb{P}^x$-a.s. for $x \in E \setminus E_M$, while the second one is entailed by Proposition \ref{prop 3.1}.
\end{proof}

\vspace{0.3cm}
\noindent
{\bf Quasimartingales under time change.} 
Let $A$ be a perfect continuous additive functional of $X$ ($AF$) and $F = {supp}(A)$ its fine support. 
Then the inverse $\tau_t$ of $A_t$ defined
$$
\tau_t(\omega) := \inf\{ s : A_s(\omega) > t \},
$$
is a stopping time for each $t \geq 0$ and the process $(\tau_t)_{t \geq 0}$ is right continuous.
Set $Y_t(\omega) := X_{\tau_t(\omega)}(\omega)$, $\mathcal{G}_t 
:= \mathcal{F}_{\tau_t}$, $t \geq 0$, $\mathcal{G} = \mathop{\bigcup}\limits_{t \geq 0} \mathcal{G}_t$. 
Then the process $Y = (\Omega, \mathcal{G}, \mathcal{G}_t, Y_t, \mathbb{P}^x)$ 
is a right process on $F$ and is called the time changed process of $X$ w.r.t. $A$; 
see \cite{Sh88}, Chapter VII (more precisely, Theorem 65.9). 
We denote its resolvent by $\widehat{\mathcal{U}}$.

\begin{coro} \label{coro 3.4} 
If $u$ is a quasimartingale function for $X$ then $u|_{F}$ is a quasimartingale function for $Y$. 
Conversely, if $F = E$, then any quasimartingale function for $Y$ is a quasimartingale function for $X$.
\end{coro}

\begin{proof} 
If $u$ is a quasimartingale function for $X$, 
then by Corollary \ref{coro 2.7}, $u = u_1-u_2$ with $u_1, u_2 \in E(\mathcal{U})$ and real-valued. 
But $E(\mathcal{U})|_F \subset E(\widehat{\mathcal{U}})$ (see \cite{Sh88}, 65.12), so $u|_F$ 
is a quasimartingale function for $Y$ by the same Corollary \ref{coro 2.7}. 
If $F = E$, the result follows by same arguments, since is this case, 
$E(\mathcal{U}) = E(\widehat{\mathcal{U}})$; cf. \cite{Sh88}, 65.13.
\end{proof}

The $\alpha$-quasimartingales are not preserved by time change, 
since $E(\mathcal{U}^{\alpha}) \not\subset E(\overline{\mathcal{U}}^{\alpha})$, $\alpha > 0$, in general.
However, the following result holds.

\begin{prop} \label{prop 3.5} 
If $u$ is an $\alpha$-quasimartingale function of $X$ for some $\alpha \geq 0$, 
then the process $(e^{-\alpha \tau_t}u(Y_t))_{t \geq 0}$ is a $\mathbb{P}^x$-quasimartingale w.r.t. 
the filtration $(\mathcal{G}_t)_{t \geq 0}$ for all $x \in F$.
\end{prop}

\begin{proof} 
If $u$ is an $\alpha$-quasimartingale function for $X$, then by Corollary \ref{coro 2.7}, $u = u_1 - u_2$ with $u_1, u_2 \in E(\mathcal{U}^{\alpha})$ finite on $E$. 
By Doob stopping theorem we have that
$$
\mathbb{E}^x\{e^{-\alpha \tau_t}u_i(X_{\tau_t})\} \leq u_i(x), \; x \in E, \; t \geq 0, \; i = \overline{1,2}.
$$
On the other hand, $(\alpha \tau_t)_{t \geq 0}$ is a perfect right-continuous $AF$ of $Y$, 
hence $(e^{-\alpha \tau_t})_{t \geq 0}$ is an exact and perfect $MF$ of $Y$; see \cite{Sh88}, 54.11. 
Let $(Q_t)_{t \geq 0}$ be the transition function of the process $Y$ killed by $(e^{-\alpha \tau_t})_{t \geq 0}$. 
Then
$$
Q_tu_i|_{F}(x) = \mathbb{E}^x\{e^{-\alpha \tau_t}u_i(X_{\tau_t})\} \leq u_i(x), \; x \in F,
$$
which means that $u_i|_F$ is $(Q_t)$-excessive, hence $V^{(Q_t)}(u|_F) < \infty$ (cf. Theorem \ref{thm 2.6}, ii)). 
The result now follows since 
${Var}^{\mathbb{P}^x}((e^{-\alpha \tau_t}(X_{\tau_t}))_{t \geq 0}) = V^{(Q_t)}(u|_F)(x)$ by Proposition 2.1.
\end{proof}

\vspace{0.3cm}
\noindent
{\bf Quasimartingales under Bochner subordination.} 
Assume that $X$ is transient and let $\mu:=(\mu_t)_{t \geq 0}$ 
be a vaguely continuous convolution semigroup of subprobability measures on $\mathbb{R}_+$. 
Define the {\it subordinate} $(P_t^{\mu})_{t\geq 0}$ of $(P_t)_{t \geq 0}$ by
$$
P_t^{\mu}f := \int_0^{\infty} P_sf \mu_t(ds) \quad {\rm for \; all} \; f \in bp\mathcal{B},
$$
whose resolvent is denoted by $\mathcal{U}^{\mu} := (U_{\alpha}^{\mu})_{\alpha \geq 0}$.
By \cite{Lu13}, Theorem 3.3, $(P_t^{\mu})_{t \geq 0}$ is the transition function of a right process $X^{\mu}$ on $E$. Moreover, $E(\mathcal{U}) \subset E(\mathcal{U}^{\mu})$, hence we have the following result.

\begin{coro} \label{coro 3.6} 
Any quasimartingale function for $X$ is a quasimartingale function for $X^{\mu}$.
\end{coro}

\vspace{0.2cm}

\noindent
{\bf Example.} Recall that a sub-Markovian resolvent of kernels $\mathcal{V} = (V_{\alpha})_{\alpha}$ 
is said to be $S$-subordinate to $\mathcal{U}$ if $E(\mathcal{U}) \subset E(\mathcal{V})$; see \cite{HmHm09} and \cite{Si99}.
 
By Corollary \ref{coro 2.7}, it follows that the class of 
quasimartingale functions for $X$ is inherited by any right process whose resolvent is $S$-subordinate to $\mathcal{U}$. 
We remark that killing, time change, Bochner subordination, 
and any combination of them, may be regarded as $S$-subordinations w.r.t. 
$\mathcal{U}$, hence the quasimartingale functions for $X$ are preserved under such transformations.
We emphasize that since the killing, time change, 
and Bochner subordination transformations do not commute in general, the order of any combination of them is relevant. 
We illustrate such a situation by looking at (Bochner) subordinate killed and killed subordinate Brownian motion. 
We follow \cite{SoVo03}; see also \cite{HmJa14}, Example 7.

Let $(X_t)_{t \geq 0}$ be a $d$-dimensional Brownian motion on $\mathbb{R}^d$ 
and $(\xi_t)_{t \geq 0}$ an $\alpha$-stable subordinator starting at 0, $\alpha \in (0,1)$. 
Let $Y_t = X_{\xi_t}$ be the right process whose transition function is the subordinate 
$(P_t^{\mu})_{t \geq 0}$ of $(P_t)_{t \geq 0}$ by means of the convolution semigroup $\mu$ induced by $(\xi_t)_{t \geq 0}$.
The generator of $Y$ is $-(-\Delta)^{\alpha}$, the fractional power of the negative Laplacian.
Let now $D \subset \mathbb{R}^d$ be a domain and denote by $Y^D$ the killed upon leaving $D$, 
which is a right process obtained by killing $Y$ with the exact $MF$ $M_t = 1_{[0, T_{D^c})}(t)$, $t \geq 0$, 
where $T_{D^c}(\omega) := \inf\{ t > 0 \; | \; Y_t(\omega) \in D^c \}$.

Changing the order of transformations, let $Z$ be the right process obtained by first killing $X$ upon leaving $D$ and then subordinating the killed Brownian motion by means of $\mu$. 
The generator of $Z$ is $-(-\Delta|_D)^{\alpha}$.
As remarked in \cite{HmJa14}, $Z$ is $S$-subordinate to $Y^D$, hence:

\begin{coro} \label{coro 3.7} 
Any quasimartingale function for $Y^D$ is a quasimartingale function for $Z$.
\end{coro}

\section{Criteria for quasimartingale functions}

In this section we present some sufficient conditions for a function to be an $\alpha$-quasimartingale function. 
In the first part we develop the study from the resolvent point of view, 
while in the last part we place ourselves in an $L^p$-context ($C_0$-semigroups and infinitesimal generators) 
with respect to a sub-invariant measure.
\vspace{0.3cm}

\noindent
{\bf A resolvent approach.} Again, we deal with a fixed right Markov process 
$X = (\Omega, \mathcal{F}, \mathcal{F}_t, X_t, \mathbb{P}^x)$ on $(E, \mathcal{B})$, with transition function $(P_t)_{t \geq 0}$ 
and resolvent $\mathcal{U} = (U_{\alpha})_{\alpha > 0}$. 

The main result of this subsection is the following.
\begin{prop} \label{prop 4.1}
 Let $u$ be a real-valued $\mathcal{B}$-measurable finely continuous function. 

i) Assume there exist $\alpha \geq 0$ and $c \in p\mathcal{B}$ such that
$$
U_{\alpha}(|u| + c) < \infty, \quad \mathop{\lim\sup}\limits_{t \to \infty} P_t^{\alpha}|u| < \infty, \quad |P_t u - u| \leq c t, t \geq 0,
$$
and the functions $t \mapsto P_t(|u| + c)(x)$ are Riemann integrable.  
Then $u$ is an $\alpha$-quasimartingale function.

ii) Assume there exist $\alpha \geq 0$, $c \in p\mathcal{B}$ such that 
$$
|P_t u- u| \leq c t, t \geq 0, \quad \mathop{\sup}\limits_{t \in \mathbb{R}_+} P_t^{\alpha}(|u| + c) = : b < \infty.
$$ 
Then $u$ is a $\beta$-quasimartingale function for all $\beta > \alpha$.

iii) Assume there exists $x_0 \in E$ such that for some $\alpha \geq 0$
$$
U_{\alpha}(|u|)(x_0) < \infty, \quad U_{\alpha}(|P_t u - u|)(x_0) \leq {const} \cdot t, \; t \geq 0.
$$
Then $[V^{\beta}(u) < \infty] \neq \emptyset$ and if $\mathcal{U}$ is irreducible 
(e.g. strong Feller and topologically irreducible) then $u$ is a $\beta$-quasimartingale function for all $\beta > \alpha$.
\end{prop}

\begin{proof}
Assume that the conditions in i) are satisfied and let 
$$
\tau_n := \left\{ \dfrac{k}{2^n} \; : \; 0 \leq k \leq n \cdot 2^n \right\}, \; n \geq 1.
$$

\noindent
Clearly, $(\tau_n)_{n \geq 1}$ is an admissible sequence of partitions of $\mathbb{R}_+$ (see Definition \ref{defi 2.5}), hence, by Theorem \ref{thm 2.6}, iii), we have to check that $\mathop{\lim}\limits_{n}V_{\tau_n}^{\alpha}(u) < \infty$.
First, note that by hypotheses,
$$
|P_t^{\alpha} u - u| \leq |P_t u|(1 - e^{-\alpha t}) + ct
\leq ({const} |P_t u| + c)t
\leq {const} (|u| + c)t
$$
for all $t$ small enough. Therefore,

\vspace{0.2cm}
\qquad \; 
$\mathop{\lim}\limits_{n} V_{\tau_n}^{\alpha}(u) = \mathop{\lim}\limits_{n}\left\{ \mathop{\sum}\limits_{k=1}^{n 2^n - 1} P_{\frac{k-1}{2^n}}^{\alpha} | P_{\frac{1}{2^n}}^{\alpha} u - u | + 
P_{n 2^n}^{\alpha}|u|\right\}$

\vspace{0.2cm}
\qquad \qquad \quad \quad \;\;\; 
$\leq \mathop{\lim\sup}\limits_{n} \left\{ \mathop{\sum}\limits_{k=1}^{n 2^n - 1} P_{\frac{k-1}{2^n}}^{\alpha} |P_{\frac{1}{2^n}}^{\alpha}u - u| \right\} + \mathop{\lim\sup}\limits_{n} P_{n \cdot 2^n}^{\alpha}|u|$.

\vspace{0.2cm}
\noindent
By hypothesis, $\mathop{\lim\sup}\limits_{n} P_{n \cdot 2^n}^{\alpha}|u| < \infty$. 
As for the other term, we have
$$
\mathop{\lim\sup}\limits_{n} \left\{ \mathop{\sum}\limits_{k=1}^{n 2^n - 1} P_{\frac{k-1}{2^n}}^{\alpha} |P_{\frac{1}{2^n}}^{\alpha} u - u| \right\} \leq
{const} \cdot \mathop{\lim\sup}\limits_{n} \{ \dfrac{1}{2^n} \mathop{\sum}\limits_{k=1}^{n 2^n -1} P_{\frac{k-1}{2^n}}^{\alpha}(|u| + c)\} 
$$

\qquad \qquad \qquad \qquad \qquad \qquad \qquad \qquad 
$= {const} \cdot \int_0^{\infty}e^{-\alpha t}P_t(|u| + c)dt$

\vspace{0.2cm}

\qquad \qquad \qquad \qquad \qquad \qquad \qquad \qquad 
$= {const} \cdot U_{\alpha}(|u| + c) < \infty$.

\vspace{0.2cm}

ii). Let $\beta > 0$.
Similarily to the above computations and noticing that 
$\lim\limits_{t \to \infty}P_t^\beta|u|=0$,
$$
\lim\limits_nV_{\tau_n}^\beta (u) \leq const \cdot \limsup\limits_n \sum\limits_{k=1}^{n2^n-1}P^\beta_{\frac{k-1}{2^n}}(|u| + c)\frac{1}{2^n}
$$

\qquad \qquad \qquad \qquad \qquad \qquad $\leq const \cdot b \; \limsup\limits_n \sum\limits_{k=1}^{n2^n-1}e^{-(\beta-\alpha){\frac{k-1}{2^n}}}\frac{1}{2^n}$

\vspace{0.2cm}
\qquad \qquad \qquad \qquad \qquad \qquad $= const \cdot b \int\limits_0^\infty e^{-(\beta - \alpha)t dt} < \infty$.

\vspace{0.2cm}

iii). Let $\beta > \alpha$. 
Once we show that $[V^{\beta}(u) < \infty] \neq \emptyset$, the second assertion follows by Corollary \ref{coro 2.9}. 
Taking into account Theorem \ref{thm 2.6}, iii), we will show that 
$U_{\alpha}(\mathop{\lim}\limits_{n}V_{\tau_n}^{\beta}(u))(x_0) < \infty$. 
Notice first that $\delta_{x_0}\circ U_{\alpha}$ 
is an $\alpha$-sub-invariant measure for $(P_t)_t$, i.e. 
$U_{\alpha}(P_t^{\alpha}f)(x_0) \leq U_{\alpha} f(x_0)$ for all $f \in p\mathcal{B}$. 
Employing this property and using the hypotheses,

\vspace{0.2cm}

$U_{\alpha}(\mathop{\lim}\limits_{n}V_{\tau_n}^{\beta}(u))(x_0) = \mathop{\lim}\limits_{n} U_{\alpha}(V_{\tau_n}^{\beta}(u))(x_0)$

\vspace{0.2cm}

\qquad \qquad \qquad \quad \;\;\; 
$= \mathop{\lim}\limits_{n}\left\{ \mathop{\sum}\limits_{k = 1}^{n\cdot 2^n -1} U_{\alpha}( P^{\beta}_\frac{k-1}{2^n} | u - P^{\beta}_\frac{1}{2^n}u|)(x_0) + U_{\alpha}(P_{n}(|u|))(x_0)\right\}$

\vspace{0.2cm}

\qquad \qquad \qquad \quad \;\;\; 
$\leq \mathop{\lim\sup}\limits_{n} \mathop{\sum}\limits_{k=1}^{n 2^n -1}e^{-(\beta - \alpha) \frac{k-1}{2^n}} U_{\alpha} | u - P^{\beta}_\frac{1}{2^n}u |(x_0) $

\qquad \qquad \qquad \qquad \quad $+\mathop{\lim\sup}\limits_{n} e^{-(\beta - \alpha)n} U_{\alpha}(|u|)(x_0)$

\vspace{0.2cm}

\qquad \qquad \qquad \quad \;\;\; $\leq {const} \; \cdot \int_{0}^{\infty} e^{-(\beta - \alpha)t} dt < \infty$.
\end{proof}

\noindent
{\bf An $L^p$-approach.} Throughout this subsection we assume that 
$\mu$ is a $\sigma$-finite sub-invariant measure for $(P_t)_{t 
\geq 0}$. 
Hence $(P_t)_{t \geq 0}$ and $\mathcal{U}$ 
extend to strongly continuous semigroup resp. resolvent family of contractions on $L^p(\mu)$, $1 \leq p < \infty$. 
The corresponding generators $({\sf L}_p, D({\sf L}_p) \subset L^p(\mu))$ are defined by
$$
D({\sf L}_p) = \{ U_{\alpha} f | f \in L^p(\mu) \},
$$
$$
{\sf L}_p(U_{\alpha} f) = \alpha U_{\alpha} f - f \quad {\rm for \; all} \; f \in L^p(\mu), 1 \leq p < \infty,
$$
with the remark that this definition si independent of $\alpha > 0$.

The corresponding notations for the dual structure are 
$\widehat{P}_t$ and $(\widehat{\sf L}_p, D(\widehat{\sf L}_p))$, 
and note that the adjoint of ${\sf L}_p$ is $\widehat{\sf L}_{p^\ast}$; $\frac{1}{p} + \frac{1}{p^\ast}=1$.

\vspace{0.2cm}

The main results of Section 2, namely Theorem \ref{thm 2.6} and its Corollary \ref{coro 2.7}, 
can be reformulated in the $L^p(\mu)$ context. Although the proofs follow the same main ideas, 
they become simpler due to the strong continuity of $(P_t)_{t \geq 0}$ on $L^p(\mu)$. 
Because we are mainly interested in the situation when $V(u) < \infty$ on $E$ 
except some negligible set, but also for simplicity, we present below the $L^p$-version of Corollary \ref{coro 2.7} only.

\begin{prop} \label{prop 4.2}
The following assertions are equivalent for a $\mathcal{B}$-measurable function $u \in \mathop{\bigcup}\limits_{1 \leq p < \infty} L^p(\mu)$.

\vspace{0.2cm}

i) $u(X)$ is a $\mathbb{P}^x$-quasimartingale for $\mu$-a.e. $x \in E$.

\vspace{0.2cm}

ii) $V(u) < \infty$ $\mu$-a.e.

\vspace{0.2cm}

iii) For an admissible sequence of partitions of 
$(\tau_n)_{n \geq 1}$ of $\mathbb{R}_+$, $\mathop{\sup}\limits_{n} V_{\tau_n}(u) < \infty$ $\mu$-a.e.

\vspace{0.2cm}

iv) There exist $u_1, u_2 \in E(\mathcal{U})$ finite $m$-a.e. such that $u = u_1 - u_2$ $\mu$-a.e.
\end{prop}

\begin{proof} 
We prove only iii) $\Rightarrow$ iv), just to point out the benefit of the strong continuity of $(P_t)_{t \geq 0}$ on $L^p(\mu)$.

As in the proof of Theorem \ref{thm 2.6}, iii), if we define $\widetilde{u}_1$ and $\widetilde{u}_2$ $\mu$-a.e. by
$$
\widetilde{u}_1 = \mathop{\sup}\limits_{n} u^{\tau_n}_1 = \mathop{\sup}\limits_{n} \left\{ \mathop{\sum}\limits_{i=1}^{n} P_{t_{i-1}}(u - P_{t_i - t_{i-1}}u)^+ + P_{t_n}(u^+) \right\},
$$
$$
\widetilde{u}_2 = \mathop{\sup}\limits_{n} u^{\tau_n}_2 = 
\mathop{\sup}\limits_{n} \left\{ \mathop{\sum}\limits_{i=1}^{n} P_{t_{i-1}}(u - P_{t_i - t_{i-1}}u)^- + P_{t_n}(u^-) \right\},
$$
then $\widetilde{u}_i$ are finite $m$-a.e. and one can show that
$P_r \widetilde{u}_i \leq \widetilde{u}_i$ for all 
$r \in \mathop{\bigcup}\limits_{n \geq 1}\tau_n, i = \overline{1,2}$,
and $u = \widetilde{u}_1 - \widetilde{u}_2$ $\mu$-a.e.

If $t \in [0, \infty)$ and $\mathop{\bigcup}\limits_{n \geq 1}\tau_n \supset (t_k)_k \searrow t$ 
then for $i =\overline{1,2}$ and $\mu$-a.e.,
$$
P_t\widetilde{u}_i = \mathop{\sup}\limits_{n}P_t u^{\tau_n}_i = 
\mathop{\sup}\limits_{n} \mathop{\lim}\limits_{k} P_{t_k}u^{\tau_n}_i \leq \mathop{\lim}\limits_{k} P_{t_k} \widetilde{u}_i
\leq \widetilde{u}_i,
$$
with the remark that the second holds $\mu$-a.e. because 
$u^{\tau_n}_i \in L^p(\mu)$ and $(P_t)_{t \geq 0}$ is strongly continuous.
Then, cf. e.g. \cite{BeCiRo15}, Proposition 2.4, 
there exist two $\mathcal{B}$-measurable functions 
$u_1, u_2 \in E(\mathcal{U})$ s.t. $\widetilde{u}_i = u_i$ $\mu$-a.e., and finally, $u = u_1 - u_2$ $\mu$-a.e.
\end{proof}

\begin{rem} \label{rem 4.3} i) We point out that for the proof of 
Proposition \ref{prop 4.2} we did not really used the fact that $\mu$ is sub-invariant, 
but just that $(P_t)_t$ is strongly continuous on $L^p(\mu)$.
In particular, Proposition \ref{prop 4.2} remains true for $u \in L^\infty(\mu)$ 
if we regard $(P_t)_t$ as a strongly continuous semigroup on $L^1(\widehat{U}_1 f \cdot \mu)$ for some $0<f\in L^1(\mu)$.

\vspace{0.2cm}

ii) If $u$ is a $\mathcal{B}$-measurable finely continuous function from 
$\mathop{\bigcup}\limits_{1 \leq p \leq \infty} L^p(\mu)$ 
satisfying any of the equivalent assertions in Proposition \ref{prop 4.2}, then the decomposition 
$u = u_1 - u_2$ with $u_1, u_2 \in E(\mathcal{U})$ holds q.e.
\end{rem}

Now, we focus our attention on a class of $\alpha$-quasimartingale 
functions which arises as a natural extension of $D({\sf L}_p)$. 
First of all, it is clear that any function $u \in D({\sf L}_p)$, $1 \leq p < \infty$, 
has a representation $u = U_{\alpha} f = U_{\alpha}(f^+) - U_{\alpha}(f^-)$ with $U_{\alpha}(f^{\pm}) \in E(\mathcal{U}^{\alpha}) \cap L^p(\mu)$. 
In particular, $u$ has an $\alpha$-quasimartingale version for all $\alpha > 0$.
Moreover, $\| P_t u - u \|_p = \left\| \int_0^t P_s{\sf L}_p u ds \right\|_p \leq t \| {\sf L}_p u \|_p$. 
Conversely, if $1 < p < \infty$, 
$u \in L^p(\mu)$, and $\| P_t u - u \|_p \leq {const} \cdot t$, $t \geq 0$, then due to the reflexivity of $L^p$ we have that the family $ \{\frac{P_t u - u}{t}\}_{t\geq 0}$ is weakly relatively compact, and by duality one can easily check that any weakly convergent subsequence $ (\frac{P_{t_n} u - u}{t_n})_{t_n \to 0}$ has the same limit. Therefore $\frac{P_t u - u}{t}$ is weakly convergent to a limit from $L^p(\mu)$ as $t$ tends to $0$, and by \cite{Sa99}, Lemma 32.3, it is strongly convergent and $u \in D({\sf L}_p)$.
But this is no longer the case if $p = 1$, and in general, 
$\| P_t u - u \|_1 \leq {const} \cdot t$ does not imply $u \in D({\sf L}_1)$.
However, this last condition on $L^1(\mu)$ is still sufficient to guarantee that 
$u$ is an $\alpha$-quasimartingale function.
In fact, the following general characterization holds.

\begin{prop} \label{prop 4.4} 
Let $1 \leq p < \infty$ and suppose $\mathcal{A} \subset \{ u \in L^{p^{\ast}}_+(\mu) : \| u \|_{p^{\ast}} \leq 1 \}$, $\widehat{P}_s\mathcal{A} \subset \mathcal{A}$ for all $s \geq 0$, and 
$E = \mathop{\bigcup}\limits_{f \in \mathcal{A}}{\rm supp}(f)$ $\mu$-a.e. 
Then the following assertions are equivalent for $u \in L^p(\mu)$.

\vspace{0.2cm}

i) $\sup\limits_{f \in \mathcal{A}}\int_E |P_tu - u| f d\mu \leq {const} \cdot t$ for all $t \geq 0$. 

\vspace{0.2cm}

ii) For every $\alpha > 0$ there exist $u_1, u_2 \in E(\mathcal{U}^\alpha)$ which satisfy i), 
$\mathop{\sup}\limits_{f \in \mathcal{A}} \int_E (u_1 + u_2) f d\mu < \infty$, and $u = u_1 - u_2$ $\mu$-a.e.
\end{prop}

\begin{proof}
Since ii) $\Rightarrow$ i) is clear, let us prove the other implication. 
Assume that $u$ satisfies i).
Then taking $\widehat{P}_s f$ instead of $f$ in condition i) we get for all $s, t \geq 0$
$$
\int_E P_s^{\alpha} |P_t^{\alpha} u- u| f d\mu \leq \int_E P_s^{\alpha}|P_tu - u| f d\mu + 
\int_E P_s^{\alpha}|P_tu - P_t^{\alpha}u| f d\mu 
$$

\qquad \qquad \qquad \qquad \qquad \quad \;\;\; $\leq \left[ {const} \cdot t + 
(1 - e^{-\alpha t})\| u \|_p \| f \|_{p^{\ast}} \right] e^{-\alpha s}$

\vspace{0.2cm}

\qquad \qquad \qquad \qquad \qquad \quad \;\;\; $\leq {const} \cdot t e^{-\alpha s}$.

\vspace{0.2cm}

Let now $\tau_n := \left\{ \dfrac{k}{2^n} : 0 \leq k \leq n 2^n \right\}$, $n \geq 1$. Then, for $\alpha > 0$

\qquad \qquad \; 
$\int_E \mathop{\sup}\limits_{n} V_{\tau_n}^{\alpha}(u) f d\mu = 
\mathop{\lim}\limits_{n} \mathop{\sum}\limits_{k=1}^{n 2^n} \int_E P_\frac{k-1}{2^n}^{\alpha} |P_\frac{1}{2^n}^{\alpha} u - u| f d\mu$

\vspace{0.2cm}

\qquad \qquad \qquad \qquad \qquad \quad \;\;\;\;\;
$\leq {const} \cdot \mathop{\lim}\limits_{n} \mathop{\sum}\limits_{k=1}^{n 2^n} e^{-\alpha\frac{k-1}{2^n}} \frac{1}{2^n} $

\vspace{0.2cm}

\qquad \qquad \qquad \qquad \qquad \quad \;\;\;\;\;$= {const} \cdot \int_0^{\infty} e^{-\alpha t} dt < \infty$

\vspace{0.2cm}

\noindent for all $f \in \mathcal{A}$. 
Hence $\mathop{\sup}\limits_{n} V_{\tau_n}^{\alpha}(u) < \infty$ $\mu$-a.e. 
and by Proposition \ref{prop 4.2} we have that $u = u_1 - u_2$ $\mu$-a.e. with $u_1, u_2 \in E(\mathcal{U}^{\alpha})$. 
Moreover, inspecting the way $u_1$ and $u_2$ have been constructed, 
we have that $u_1 + u_2 = \mathop{\sup}\limits_{n} V_{\tau_n}^\alpha(u)$ $\mu$-a.e., hence 
$\mathop{\sup}\limits_{f \in \mathcal{A}} \int_E (u_1 + u_2) f d\mu < \infty$.
Moreover, for $r \in \mathop{\bigcup}\limits_{n \geq 1} \tau_n$ and $i=\overline{1,2}$, 
$$
u_i = \mathop{\lim}\limits_{n} \{ \mathop{\sum}\limits_{k=1}^{r2^n} P_{\frac{k-1}{2^n}}^{\alpha}(u - P_\frac{1}{2^n}^{\alpha}u)^{\pm} + P_r^{\alpha}(u - P_\frac{1}{2^n}^{\alpha}u)^{\pm} 
$$

\qquad \qquad \qquad \qquad \qquad 
$+ \mathop{\sum}\limits_{i=1}^{n2^n} P_r^{\alpha}P_\frac{i-1}{2^n}^{\alpha}( u - P_\frac{1}{2^n}^{\alpha}u)^{\pm} + 
P_r^{\alpha}P_{n - r}^{\alpha}(u^{\pm})\}
$

\vspace{0.2cm}
\qquad \qquad \qquad \quad \quad \; 
$= \mathop{\lim}\limits_{n} \left\{\mathop{\sum}\limits_{k=1}^{r2^n} P_\frac{k-1}{2^n}^{\alpha}(u - P_\frac{1}{2^n}^{\alpha}u)^{\pm} + P_r^{\alpha}(u - P_\frac{1}{2^n}^{\alpha}u)^{\pm}\right\} + P_r^{\alpha} u^i$.

\noindent
Therefore

\;\;\;\;\;\; \;\; \; $\int_E |u_i - P_r^{\alpha}u_i| f d\mu \leq \mathop{\lim}\limits_{n} \mathop{\sum}\limits_{k=1}^{r 2^n} 
\int_E P_\frac{k-1}{2^n}^{\alpha}|u - P_\frac{1}{2^n}^{\alpha}u| f d\mu$

\vspace{0.2cm}
\qquad \qquad \qquad \qquad $\leq {const} \cdot \int_0^r e^{-\alpha t} dt$ 

\vspace{0.2cm}
\qquad \qquad \qquad \qquad $= {const} \cdot r$

\vspace{0.2cm}
\noindent for all $f \in \mathcal{A}$, $i = \overline{1,2}$, $r \in \mathop{\bigcup}\limits_{n \geq 1}\tau_n$, 
where the above constant is independent of 
$f \in \mathcal{A}$, $i = \overline{1,2}$, and $r \in \mathop{\bigcup}\limits_{n \geq 1}\tau_n$.

We claim that $\int_E (u_i - P_t^{\alpha}u_i)f d\mu \leq {\rm const} \cdot t$ 
for all $t \geq 0$, $i = \overline{1,2}$, and $f \in \mathcal{A}$. 
Since the desired inequality holds for all 
$r \in \mathop{\bigcup}\limits_{n \geq 1}\tau_n$ and $0 \leq u_i - P_r^{\alpha}u_i \leq u_i$, 
by dominated convergence it is sufficient to show that for each $f \in \mathcal{A}$, 
$P_{r_k}^{\alpha}u_i$ converges $f\cdot\mu$-a.e. on a subsequence to 
$P_t^{\alpha}u_i$, whenever $\mathop{\bigcup}\limits_{n \geq 1}\tau_n \ni r_k \mathop{\searrow}\limits_{k} t \geq 0$.
To see this, let $\nu := \widehat{U}_{\alpha}f \cdot \mu$ and note that $u_i \in L^1(\nu)$. 
Since $\nu$ is a sub-invariant measure for $(P_t^{\alpha})_{t \geq 0}$ 
we have that $(P_t^{\alpha})_{t \geq 0}$ is strongly continuous on $L^1(\nu)$, 
hence if $\mathop{\bigcup}\limits_{n \geq 1}\tau_n \ni r_k \searrow t \geq 0$ 
it follows that on a subsequence, $(P_{r_k}^{\alpha}u_i)_{k \geq 1}$ converges $\nu$-a.e. to $P_t^{\alpha}u_i$. 
Since $f\cdot \mu \ll \nu$ we obtain that the above convergence holds $f \cdot \mu$-a.e. 
So,
$$
\int_E|u_i - P_t^{\alpha}u_i|f d\mu \leq { const} \cdot t, 
$$
and finally
$$
\int_E |u_i - P_tu_i|f d\mu \leq {const} \cdot t + (1 - e^{-\alpha t})\int_E u_i f d\mu \leq {const} \cdot t
$$
for all $t \geq 0$, $i = \overline{1,2}$, and independently on $f \in \mathcal{A}$.
\end{proof}

We can interpert condition i) from Proposition \ref{prop 4.4} in terms of the adjoint generator as follows.

\begin{prop} \label{prop 4.5} Let $p \in (1, \infty)$ and $q \in [1, \infty]$. 
The following assertions are equivalent for $u \in L^p(\mu)$.

\vspace{0.2cm}
i) $|\mu(u \; \widehat{\sf L}_{p^{\ast}}v)| \leq {const} \cdot (\|v\|_{\infty} + \|v\|_q)$ for all $v \in D(\widehat{\sf L}_{p^{\ast}})$.

\vspace{0.2cm}
ii) $u$ satisfies i) from Proposition \ref{prop 4.4} for all $\mathcal{A} = \{ v \in L^{p^\ast}(\mu) :  \|v \|_\infty + \| v\|_q \leq 1 \}$. 
\end{prop}
\qquad
\begin{proof} 
i) $\Rightarrow$ ii). Let $f \in L^{\infty}(\mu) \cap L^q(\mu)\cap L^{p^\ast}(\mu)$.
For $t \geq 0$ let $w := \dfrac{1}{t}{\rm sgn}(P_t u - u) f \in L^{p^{\ast}}(\mu)$ 
and $v := \int_0^t \widehat{P}_s w ds \in D(\widehat{\sf L}_{p^\ast})$. 
Then $\widehat{\sf L}_{p^{\ast}}v = \widehat{P}_t w - w$, $\|v\|_{\infty} + \|v\|_q \leq 2(\|f\|_{\infty} + \|f\|_q)$, and 
$$
\frac{1}{t}\int_E|P_tu - u| f d\mu = \int_E (P_tu - u)w d\mu = \int_E u(\widehat{P}_t w - w) d\mu
$$

\qquad \qquad \qquad \qquad \qquad \quad \quad \;\;\;\; 
$= \int_E u\; \widehat{\sf L}_{p^{\ast}}v d\mu \leq 2 \cdot  { const} \cdot (\|f\|_{\infty} + \|f\|_q)$.

\vspace{0.2cm}
Therefore, $\int_E |P_tu - u|f d\mu \leq { const} \cdot t$ for all $t \geq 0$ and $f \in \mathcal{A}$.

\vspace{0.2cm}
ii) $\Rightarrow$ i). If $\int\limits_E |P_t u - u|f d\mu \leq const \cdot t(\|f\|_\infty + \|f\|_q)$, 
then by replacing $f$ with ${\rm sgn}(P_tu-u)f$ we get 
$$\dfrac{1}{t}\int\limits_E u(\widehat{P}_t f - f) d\mu \leq const \cdot (\|f\|_\infty + \|f\|_q).$$ 
Now, if $f \in D(\widehat{\sf L}_{p^\ast})$ then assertion i) follows by letting $t$ tend to $0$.
\end{proof}

\noindent
{\bf Example: adding jumps to a Markov process.} 
Assume that $X$ is a standard process and $\sf N$ is a Markov kernel on $E$. 
As before, $\mu$ is a $\sigma$-finite sub-invariant measure for $\mathcal{U}$. 
We assume further that $\mu \circ {\sf N} \leq \mu$. 
It is well known that there exists a second Markov process $Y$ on $E$ 
whose infinitesimal generator is given by ${\sf Q} := {\sf L} - 1 + {\sf N}$; 
$D({\sf L}) = D({\sf Q})$; cf. \cite{Ba79} or \cite{BeSt94}, Theorem 1.8; 
see also \cite{Op16} for more general perturbations with kernels for generators of Markov processes. 
Let $\mathcal{V} = (V_{\alpha})_{\alpha}$ denote the resolvent of $Y$. 
Then $V_{\alpha} = U_{\alpha + 1} + U_{\alpha +1}{\sf N}V_{\alpha}$ and

$$
\mu(V_1f) = \mu( \mathop{\sum}\limits_{n = 0}^{\infty} U_2({\sf N}U_2)^nf) \leq 
\mathop{\sum}\limits_{n = 1}^{\infty} \dfrac{1}{2^n} \mu(f) = \mu(f)
$$
for all $f \in L^1_+(\mu)$, which means that $\mu$ is $\mathcal{V}$ - sub-invariant. 
Therefore, we can extend $\sf Q$ on $L^p(\mu)$, $1 \leq p < \infty$ by 
${\sf Q}_p := {\sf L}_p - 1 + {\sf N}_p$, $D({\sf Q}_p) = D({\sf L}_p)$,
where ${\sf L}_p$ and ${\sf N}_p$ are the corresponding $L^p(\mu)$-extensions of ${\sf L}$ and ${\sf N}$.

Let $(S_t)_{t \geq 0}$ be the transition function of $Y$. 
Since $\mu$ is $(S_t)_{t \geq 0}$-sub-invariant we have that $(S_t)_{t \geq 0}$ 
extends to a $C_0$-semigroup of contractions on $L^p(\mu)$, $1 \leq p < \infty$, 
for which we keep the same notation. 

Clearly, since $E(\mathcal{V}_{\alpha}) \subset E(\mathcal{U}_{\alpha + 1})$, 
we get by Corollary \ref{coro 2.7} that any $\alpha$-quasimartingale function 
for $Y$ is an $(\alpha + 1)$-quasimartingale function for $X$.
Also, as remarked in \cite{BeLu16}, Proposition 4.5, the spaces of 
differences of bounded functions from $E(\mathcal{U}_{\alpha + 1})$ 
and respectively from $E(\mathcal{V}_{\alpha})$ are the same.
Next, we show that the class of quasimartingale functions which are 
produced by the estimate in Proposition \ref{prop 4.4}, i) (and in Corollary 4.5, ii)) are the same for both $X$ and $Y$.

\begin{coro} \label{coro 4.6}
Let $1 \leq p < \infty$, $u \in L^p(\mu)$, and $\mathcal{A}$ be a bounded subset in $L^{p^{\ast}}$.
Then i) from Proposition \ref{prop 4.4} is satisfied w.r.t. $(P_t)_{t \geq 0}$ if and only if it is satisfied w.r.t. $(S_t)_{t \geq 0}$.
\end{coro}

\begin{proof}
The result follows easily since ${\sf Q}_p$ is a bounded perturbation of 
${\sf L}_p$, and by e.g. \cite{EnNa99}, Corollary 1.11, 
there exists a constant $c$ s.t. $\| P_t - S_t\|_p \leq t \cdot c$, $t \geq 0$. 
\end{proof}

\section{Applications to Dirichlet forms}

Let $E$ be a Hausdorff topological space with Borel $\sigma$-algebra 
$\mathcal{B}$, $\mu$ be a $\sigma$-finite measure on $\mathcal{B}$, 
and $\mathcal{E}$ be a bilinear form on $L^2(\mu)$ with dense domain 
$\mathcal{F}$; $\mathcal{E}_{\alpha}(\cdot , \cdot) = \mathcal{E}(\cdot , \cdot) + \alpha(\cdot , \cdot)$, $\alpha > 0$.

Recall that $(\mathcal{E}, \mathcal{F})$ is called a {\it coercive closed} form if:

i) $(\mathcal{E}, \mathcal{F})$ is positive definite and closed on $L^2(\mu)$.

ii) $\mathcal{E}$ satisfies the (weak) sector condition, i.e. there exists a constant $k$ s.t.
$$
|\mathcal{E}_1(u, v)| \leq k \mathcal{E}_1(u,u)^{\frac{1}{2}}\mathcal{E}_1(v,v)^{\frac{1}{2}} \; {\rm for \; all} \; u,v \in \mathcal{F}.
$$

The coercive closed form $(\mathcal{E}, \mathcal{F})$ is called a {\it Dirichlet} form if $u^+ \wedge 1 \in \mathcal{F}$ and both

iii) $\mathcal{E}(u + u^+ \wedge 1, u - u^+ \wedge 1) \geq 0$

iv) $\mathcal{E}(u - u^+ \wedge 1, u + u^+ \wedge 1) \geq 0$

\noindent hold for all $u \in \mathcal{E}$. 
If only iii) is satisfied then $(\mathcal{E}, \mathcal{F})$ is called a {\it semi-Dirichlet} form.

A bilinear form $(\mathcal{E}, \mathcal{F})$ on $L^2(\mu)$ is called a 
{\it lower-bounded} (semi) Dirichlet form if there exists $\alpha > 0$ s.t. 
$(\mathcal{E}_{\alpha}, \mathcal{F})$ is a (semi) Dirichlet form. 
If $(\mathcal{E}, \mathcal{F})$ is a coercive closed form, let $(P_t)_{t \geq 0}$ 
be the $C_0$-semigroup of contractions on $L^2(m)$ associated to $\mathcal{E}$, 
whose dual is denoted by $(\widehat{P}_t)_{t \geq 0}$.
Recall that condition iii) (resp. iv)) is equivalent with the sub-Markov 
property for $(P_t)_{t \geq 0}$ (resp. $(\widehat{P}_t)_{t \geq 0}$); see \cite{MaRo92}, I.4.4. 

Adopting the notations from \cite{Fu99}, for a closed set $F \subset E$ we set:

\vspace{0.2cm}
\qquad $\mathcal{F}_F = \{v \in \mathcal{F} \; : \; v = 0 \; m\mbox{-a.e. on} \; E \setminus F \}$, 

\vspace{0.2cm}
\quad \; $\mathcal{F}_{b, F} = \{v \in \mathcal{F}_{F} \; : \; v \in L^\infty(\mu)\}$.

An increasing sequence of closed sets $(F_n)_{n \geq 1}$ 
is called an {\it $\mathcal{E}$-nest} if $\mathop{\bigcup}\limits_{n = 1}^{\infty} \mathcal{F}_{F_n}$ is $\mathcal{E}_1$-dense in $\mathcal{F}$.
An element $f \in \mathcal{F}$ is called $\mathcal{E}$-quasi-continuous 
if there exists a nest $(F_n)_{n \geq 1}$ such that $f|_{F_n}$ is continuous for each $n \geq 1$.

A (semi) Dirichlet form $(\mathcal{E}, \mathcal{F})$ on $L^2(\mu)$ 
is called {\it quasi-regular} if there exist an $\mathcal{E}$-nest consisting of compact sets, 
an $\mathcal{E}_1$-dense subset of $\mathcal{F}$ 
whose elements admit $\mathcal{E}$-quasi-continuous versions, 
and a countable family of $\mathcal{E}$-quasi-continuous elements from $\mathcal{F}$ 
which separates the points of $\mathop{\bigcup}\limits_{n=1}^{\infty} E_n$ for a certain $\mathcal{E}$-nest $(E_n)_{n \geq 1}$.
It is well known that the quasi-regularity property is a necessary and sufficient condition 
for a semi-Dirichlet form to be (properly) associated to a $\mu$-tight special standard process $X$ 
(i.e. the semigroup $(P_t)_t$ of $(\mathcal{E}, \mathcal{F})$ is generated by the transition function of $X$); 
see \cite{MaOvRo95} or \cite{MaRo92} for details.
On the other hand, it was shown in \cite{BeBoRo06} 
that for any semi-Dirichlet form $(\mathcal{E}, \mathcal{F})$ 
on a Lusin measurable space $E$, one can always find a larger space $E_1$ s.t. 
$E_1 \setminus E$ has measure zero and $(\mathcal{E}, \mathcal{F})$ regarded on $E_1$ becomes quasi-regular.

Hereinafter, all of the forms are assumed to be quasi-regular, 
in particular every element $u \in \mathcal{F}$ admits a quasi-continuous version denoted by $\widetilde{u}$.

In the sequel we  often appeal to the following well known decompositions for the elements of the domain $\mathcal{F}$:

\noindent 
{\it Ortogonal decomposition via hitting distribution.} 
For a nearly Borel set $A \subset E$ and a quasi-continuous function $u \in \mathcal{F}$ 
we define the $\alpha$-order hitting distribution 
$R_\alpha^{A^c} u(x) := \mathbb{E}^x[e^{-\alpha T_{A^c}}u(X_{T_{A^c}})]$, $\alpha > 0$.
Then $R_\alpha^{A^c} u \in \mathcal{F}$ is quasi-continuous, 
$u-R_\alpha^{A^c} u \in \mathcal{F}_{A}$, 
and $\mathcal{E}_{\alpha}(R_\alpha^{A^c} u,v) = 0$ for all $v\in \mathcal{F}_A$.
When $\mathcal{E}$ is a Dirichlet form, $\widehat{R}_\alpha^{A^c} u$ 
may be defined analogously, replacing $X$ with the dual process $\widehat{X}$.
When $\mathcal{E}$ is merely a semi-Dirichlet form, the existence of the dual process is more delicate, 
and for simplicity we prefere to define $\widehat{R}_\alpha^{A^c} u$ 
as the unique element from $\mathcal{F}$ such that 
$u-\widehat{R}_\alpha^{A^c} u \in \mathcal{F}_{A}$ and $\mathcal{E}_{\alpha}(v,R_\alpha^{A^c} u) = 0$ 
for all $v\in \mathcal{F}_A$; see e.g. \cite{Os13}, Section 3.5.

\noindent
{\it Fukushima's decomposition.} (see \cite{MaRo92}, Chapter VI, Theorem 2.5, or \cite{FuOsTa11}) 
For each $u \in \mathcal{F}$ there exist a martingale additive functional of finite energy 
$(M_t)_{t \geq 0}$ (MAF) and a continuous additive functional 
$(N_t)_{t \geq 0}$ of zero energy s.t. $\widetilde{u}(X) - \widetilde{u}(X_0) = M + N$; 
we denote by $|N|_t$ the variation of $N$ on $[0,t]$.

\vspace{0.3cm}

For the rest of this section our aim is to explore conditions for an element $u \in \mathcal{F}$ 
(or more generally in $\mathcal{F}_{loc}$) ensuring that $\widetilde{u}(X)$ 
is a $\mathbb{P}^x$-semi($\alpha$-quasi)martingale q.e. $x \in E$; 
in this case we shall say shortly that {\it $\widetilde{u}(X)$ is a semi($\alpha$-quasi)martingale}. 
 
Going back to Proposition \ref{prop 4.4} and Corollary \ref{prop 4.5}, 
we note that the sub-Markov property of the dual semigroup was quite helpful 
and for this reason we shall first deal with Dirichlet forms. 
However, in the end of this section we shall see that the results can be extended 
to semi-Dirichlet forms in their full generality.
It is worth to mention that all of the forthcoming criteria for 
quasimartingale functions can be directly transferred to lower bounded 
(semi) Dirichlet forms whose semigroups are associated to right processes, 
but for simplicity we deal only with (semi) Dirichlet forms.

\begin{thm} \label{thm 5.1} The following assertions are equivalent for an element $u \in \mathcal{F}$.

\vspace{0.2cm}
i) $|\mathcal{E}(u,v)| \leq const \cdot (\|v\|_{\infty} + \|v\|_2)$ for all $v \in \mathcal{F}_b$.

\vspace{0.2cm}
ii) For one (hence all) $\alpha > 0$ there exist $u_1, u_2 \in E(\mathcal{U}^{\alpha})$ such that 

\vspace{0.2cm}
\quad ii.1) $\sup\limits_{\|f\|_{\infty} + \|f\|_2 \leq 1}\int_E (u_1 + u_2) f d\mu < \infty$.

\vspace{0.2cm}
\quad ii.2) $u = u_1 - u_2$ $\mu$-a.e. 

\vspace{0.2cm}
\quad ii.3) $\sup\limits_{\|f\|_{\infty} + \|f\|_2 \leq 1}\int_E |P_t u_i - u_i|f d\mu \leq {const} \cdot t$, $i=1,2$.

\vspace{0.2cm}
iii) For one (hence all) $\alpha > 0$, $\widetilde{u}(X)$ is an $\alpha$-quasimartingale and 
$$
\sup\limits_{\|f\|_{\infty} + \|f\|_2 \leq 1}\mathbb{E}^{f \cdot \mu}[|N|_t] \leq {const} \cdot t.
$$
for sufficiently small $t \geq 0$.

In particular, if $u$ satisfies i) then there exists a smooth measure $\nu$ such that 
$\mathcal{E}(u,v) = \nu(\widetilde{v})$ for all $v \in \mathcal{F}_b$.
\end{thm}

\begin{proof} i) $\Rightarrow$ ii). Let $t > 0$, $f \in L^{\infty}(\mu) \cap L^2(\mu)$, 
$w := {\rm sgn}(P_tu - u)f$, and set $v := \int_0^t \widehat{P}_s w ds$. 
Then
$$
\int_E |P_tu - u| f d\mu = \int_E (P_tu - u) w d\mu = \mathcal{E}(u,v) \leq const \cdot t (\|f\|_{\infty} + \|f\|_2).
$$

By Proposition \ref{prop 4.4} (take $\mathcal{A} := \{ f \in L^2(\mu): \; \|f\|_{\infty} + \|f\|_2 \leq 1 \}$) we obtain ii).

\vspace{0.2cm}
ii) $\Rightarrow$ iii). Since $\widetilde{u}$ is quasi-continuous, we have  $\widetilde{u} = u_1 - u_2$ q.e.,  
hence the first assertion is clear (by Corollary \ref{coro 3.2} for example).

Since $(e^{-\alpha t}u_i(X_t))_{t \geq 0}$, $i = \overline{1,2}$ are right continuous supermartingales, 
by \cite{Sh88}, Section VI, there exist (uniquely) local martingales $M^i$, $M^i_0 = 0$, 
and predictable right continuous increasing and non-negative processes $A^i$ 
such that $e^{-\alpha t}u_i(X_t) - u_i(X_0) = M^i_t - A^i_t$, $t \geq 0$.

If $(T_n^i)_n$ are stopping times increasing a.s. to infinity such that the stopped processes 
$(M^i_{t \wedge T_n^i})_{t \geq 0}$ are uniformly integrable martingales, we get 
$$
\mathbb{E}^x[A^i_{t \wedge T_n^i}] = 
- \mathbb{E}^x[e^{-\alpha (t \wedge T_n)}u_i(X_{t \wedge T_n})] + u_i(x) \leq u_i(x) - e^{-\alpha t}P_tu_i(x), \; x \in E.
$$
Therefore, $\mathbb{E}^x[A_t^i] \le u_i(x) - e^{-\alpha t}P_tu_i(x)$ and if $f \in L^{\infty}(\mu) \cap L^2(\mu)$
$$
\mathbb{E}^{f\cdot \mu}[A_t^i] \leq 
\mu((u_i - e^{-\alpha t}P_tu_i)f) \leq 
\mu(|u_i - P_tu_i|f) + (1 - e^{-\alpha t})\mu(u_i \widehat{P}_t f)
$$

\qquad \qquad \quad \;\;\; $\leq {const} \cdot t(\|f\|_{\infty} + \|f\|_2)$.

\vspace{0.2cm}\noindent
Also, $e^{-\alpha t}\widetilde{u}(X_t) - \widetilde{u}(X_0) = \overline{M_t} + \overline{A_t}$, 
where $\overline{M} := M^1 - M^2$ is a local martingale 
and $\overline{A} := A^2- A^1$ is a predictable right continuous process of bounded variation.

On the other hand, since $\widetilde{u}(X)$ is an $\alpha$-quasimartingale, 
it follows that $N$, the CAF from Fukushima decomposition, is a continuous semimartingale, 
hence it is the sum of a local martingale and a continuous process with bounded variation (see e.g. \cite{Pr05}, page 131). 
But $N$ has zero energy so the quadratic variation of its martingale part is zero, hence $N$ is of bounded variation.
Then, integrating by parts,
$$
e^{-\alpha t}\widetilde{u}(X_t) - \widetilde{u}(X_0) = 
\int_0^t e^{-\alpha s}dM_s + e^{-\alpha t}N_t - \alpha \int_0^t M_se^{-\alpha s}ds - \widetilde{u}(X_0)(1 - e^{-\alpha t}).
$$
By the uniqueness of the canonical decomposition of $(e^{-\alpha t}\widetilde{u}(X_t))_{t \geq 0}$ we get that 
$$
\overline{A_t} = e^{-\alpha t}N_t - \alpha \int_0^t M_se^{-\alpha s}ds - \widetilde{u}(X_0)(1 - e^{-\alpha t}). 
$$
Therefore, 
$$
N_t = e^{\alpha t}\overline{A_t} + \alpha e^{\alpha t}\int_0^t M_se^{-\alpha s}ds + 
\widetilde{u}(X_0)e^{\alpha t}(1 - e^{-\alpha t}),
$$
and
$$
|N|_t \leq e^{\alpha t}(A_t^1 + A_t^2) + \alpha e^{\alpha t}\int_0^t |M_s|e^{-\alpha s} ds + 
|\widetilde{u}(X_0)|e^{\alpha t}(1 - e^{-\alpha t}).
$$
But by the previously obtained estimates for $\mathbb{E}^{f \cdot \mu}[A_t^i]$, we get 
$$
\mathbb{E}^{f \cdot \mu}[|N|_t] \leq {const} \cdot t e^{\alpha t}(\|f\|_{\infty} + \|f\|_2) + 
e^{\alpha t}(1 - e^{-\alpha t})\mu(f|u|) + e^{\alpha t}(1 - e^{-\alpha t})t \mathbb{E}^{f \cdot \mu}[|M_t|]
$$

\qquad \;\;\;\; $\leq {const} \cdot t (\|f\|_{\infty} + \|f\|_2)$

\vspace{0.2cm}
\noindent for conveniently small $t$, since $\mathbb{E}^{f \cdot \mu}[|M_t|] \leq \|f\|_2\mathbb{E}^{\mu}[M_t^2]$ and $M$ is of finite energy (i.e. $\lim_{t \to 0} \frac{1}{t}\mathbb{E}^{\mu}[M_t^2] < \infty$).

iii) $\Rightarrow$ i). By Revuz correspondence (see \cite{MaRo92}, Theorem 2.4), 
if $v = \widehat{U}_{\alpha}f$ for some $\alpha > 0$ and $f \in L^2(\mu) \cap L^{\infty}(\mu)$
$$
\mathcal{E}(u,v) = \mathop{\lim}\limits_{t \to 0} \displaystyle\frac{1}{t} \int_E (P_tu - u)v d\mu = 
\mathop{\lim}\limits_{t \to 0} \displaystyle\frac{1}{t} \mathbb{E}^{v \cdot \mu}[N_t] = \nu(\widetilde{v}),
$$
where $\nu$ is the signed Revuz measure associated to $N$. 
By an approximation argument, 
$|\mathcal{E}(u,v)| = |\nu(\widetilde{v})| \leq { const} \; (\|v\|_{\infty} + \|v\|_2)$ for all $v \in \mathcal{F}_b$.
\end{proof}

Further versions of Theorem \ref{thm 5.1} can be taken into account. 
For example, the following result extends Theorem 6.2 from \cite{Fu99} 
to the non-symmetric case and it can be proved in the same manner as Theorem \ref{thm 5.1}, so we omit it.

\begin{thm} \label{thm 5.2} 
The following assertions are equivalent for $u \in \mathcal{F}$.

\vspace{0.2cm}
i) $|\mathcal{E}(u,v)| \leq {const} \cdot \|v\|_{\infty}$ for all $v \in \mathcal{F}_b$.

\vspace{0.2cm}
ii) For each $\alpha > 0$, $\widetilde{u}(X)$ is an $\alpha$-quasimartingale 
and $\mathbb{E}^\mu[|N|_t] \leq {const} \cdot t$ for small $t$.

\vspace{0.2cm}
iii) There exists a smooth signed measure (the Revuz measure of $N$) $\nu$ 
such that $\nu$ is finite and $\mathcal{E}(u,v) = \nu(\widetilde{v})$ for all $v \in \mathcal{F}$.
\end{thm}

Now, we turn our attention to the situation when any of the equivalent assertions of Theorem \ref{thm 5.2} holds only locally. 
The following result extends Theorem 6.1 from \cite{Fu99} to the non-symmetric case.

\begin{thm} \label{thm 5.3} The following assertions are equivalent for $u \in \mathcal{F}$.

\vspace{0.2cm}
i) $\widetilde{u}(X)$ is a semimartingale.

\vspace{0.2cm}
ii) There exists a nest $(F_n)_{n \geq 1}$ and constants $c_n$ such that
$$
|\mathcal{E}(u,v)| \leq c_n \|v\|_{\infty} \;\; for \; all \; v \in \mathcal{F}_{b,F_n}.
$$
\end{thm}

\begin{proof} i) $\Rightarrow$ ii). As in the poof of ii) $\Rightarrow$ iii) in Theorem \ref{thm 5.1}, 
if $\widetilde{u}(X)$ is a semimartingale then $N$ (the CAF in Fukushima decomposition) 
is a continuous semimartingale of zero energy, hence it is of bounded variation. 
By \cite{MaRo92}, Theorem 2.4, $N$ is in Revuz correspondence with a signed smooth measure $\nu$, 
with an attached nest of compacts $(F_n)_{n \geq 1}$ s.t. $\nu(F_n) < \infty$.
Then just as in the proof of Theorem 5.4.2. in \cite{FuOsTa11}, 
one obtains that $\mathcal{E}(u,v) = \nu(\widetilde{v})$ for all $v \in \mathcal{F}$.

\vspace{0.2cm}
ii) $\Rightarrow$ i).
Without loss we can assume that $\mu(F_n) < \infty$.
Also, since $(F_n)_n$ is a nest we have that $\mathop{\lim}\limits_{n} T_{F_n^c} \geq \xi$ a.s. 
Due to a result of Meyer (see e.g. \cite{Pr05}, Theorem 6) it is sufficient to show that 
$(\widetilde{u}(X_t)1_{[0, T_{F_n^c})}(t))_{t \geq 0}$ is a semimartingale for each $n$ (such an argument was also employed in \cite{CiJaPrSh80}, after Theorem 4.6 ). 
On the other hand, $(e^{-t}R_1^{F_n^c}\widetilde{u}(X_t)1_{[0, T_{F_n^c})}(t))_{t \geq 0}$ 
is a difference of two right continuous supermartingales, 
so we only have to check that $(\widetilde{u} - R_1^{F_n^c}\widetilde{u})(X)$ is a semimartingale. 
But, if $v \in \mathcal{F}_b$,
$$
|\mathcal{E}_1(\widetilde{u} - R_1^{F_n^c}\widetilde{u}, v)| = 
|\mathcal{E}_1(\widetilde{u} - R_1^{F_n^c}\widetilde{u}, v - \widehat{R}_1^{ F_n^c}v)| 
= |\mathcal{E}_1(\widetilde{u}, v - \widehat{R}_1^{ F_n^c}v)| 
$$

\qquad \qquad \qquad \qquad \quad \quad \;\;\; $\leq (c_n + \int_{F_n}|u| d\mu) \|v - \widehat{R}_1^{F_n^c}v\|_{\infty}$

\vspace{0.2cm}
\qquad \qquad \qquad \qquad \quad \quad \;\;\; $\leq 2(c_n + \int_{F_n}|u| d\mu) \|v\|_{\infty}$,
 
\vspace{0.2cm}
\noindent
and by Theorem \ref{thm 5.1} it follows that $(\widetilde{u} - R_1^{F_n^c}\widetilde{u})(X)$ is a semimartingale.
\end{proof}

Recall that $(\mathcal{E}, \mathcal{F})$ is said to be {\it local} if for all pairs of elements $u,v \in \mathcal{F}$ with disjoint compact supports, it holds that $\mathcal{E}(u,v)=0$.
By \cite{MaRo92}, Chapter V, Theorem 1.5, $(\mathcal{E}, \mathcal{F})$ is local if and only if the associated process is a diffusion.

When $\mathcal{E}$ is local, Theorem \ref{thm 5.3} remains true if $u$ is assumed to be only locally in $\mathcal{F}$. 
Actually, the following even more general statement holds.

\begin{coro} \label{coro 5.4}
Assume that $(\mathcal{E}, \mathcal{F})$ is local.
Let $u$ be a real-valued $\mathcal{B}$-measurable finely continuous function and let $(v_k)_k \subset \mathcal{F}$ such that $v_k \mathop{\longrightarrow}\limits_{k \to\infty} u$ pointwise except a $\mu$-polar set and boundedly on each $F_n$. 
Further, suppose that there exist constants $c_n$ such that
$$
|\mathcal{E}(v_k, v)| \leq c_n \|v\|_{\infty} \;\; for \; all \; v \in \mathcal{F}_{b, F_n}.
$$
Then $u(X)$ is a semimartingale.
\end{coro}

\begin{proof} As we already mentioned in the proof of Theorem \ref{thm 5.3}, ii) $\Rightarrow$ i), 
it is sufficient to show that $(u - R_1^{F_n^c}u)(X)$ is a semimartingale.
Also, we saw that $|\mathcal{E}_1(\widetilde{v}_k - R_1^{F_n^c}\widetilde{v}_k, v)| \leq const \cdot \|v\|_{\infty}$ for all $v \in\mathcal{F}_b$, where the constant in the right-hand side may depend on $n$.
Now, by Theorem \ref{thm 5.1} we have that by setting $\widetilde{v}_k^n := \widetilde{v}_k - R_1^{F_n^c}\widetilde{v}_k$,
$$
\mu(\widetilde{v}_k^n(\widehat{P}_t^{1}f - f)) \leq {const} \cdot t (\|f\|_{\infty} + \|f\|_2).
$$
But by hypothesis, $R_1^{F_n^c}\widetilde{v}_k(x) = \mathbb{E}^x[e^{-T_{F_n^c}}\widetilde{v}_k(X_{T_{F_n^c}})]$ converges $\mu$-a.s. and boundedly to $R_1^{F_n^c}u$ as $k$ tends to infinity. 
So by setting $u_n := u - R_1^{F_n^c}u$ and by dominated convergence
$$
\mu((P_t^1u_n - u_n)f) = \mu(u_n(\widehat{P}_t^{ 1}f - f)) = \lim\limits_k \mu(\widetilde{v}_k^n(\widehat{P}_t^{1}f - f)) 
$$

\qquad \qquad \qquad \qquad \qquad\quad\quad\;\; $\leq {const} \cdot t (\|f\|_{\infty} + \|f\|_2) $

\vspace{0.2cm}
\noindent for all $f \in L^1(\mu)$.
Therefore, by Proposition \ref{prop 4.4} we get that $u_n(X)$ is a semimartingale.
\end{proof}

\vspace{0.2cm}

\subsection{Extensions to semi-Dirichlet forms} 
We reiterate that for the previous results of this section, where we considered only Dirichlet forms, 
it was used the fact that the adjoint semigroup $(\widehat{P}_t)_{t\geq 0}$ was sub-Markovian; 
e.g. in order to have the estimate $\|\int_0^t \widehat{P}_s f\; ds\|_\infty \leq t \|f\|_\infty$.
In this subsection we show that, as a matter of fact, the sub-Markov property 
of the adjoint semigroup is not crucial and most of the previous results remain valid for semi-Dirichlet forms.
More precisely, although in order to extend theorems \ref{thm 5.1} and \ref{thm 5.2}, i) $\Rightarrow$ ii) 
to semi-Dirichlet forms, essentially with the same proofs, it is sufficient to assume the existence of a strictly positive bounded co-excessive function, 
Theorem \ref{thm 5.3}, $ii) \Rightarrow i)$ remains true without any further assumptions, due to a standard localization procedure.
Finally, before we present the announced extensions, 
we emphasize once again that the case of lower bounded semi-Dirichlet forms follows easily, 
by working with $\mathcal{E}_\alpha$ instead of $\mathcal{E}$, for instance.
 
Hereinafter, we keep the same context and notations as before, 
but we assume that $(\mathcal{E}, \mathcal{F})$ is merely a (quasi-regular) semi-Dirichlet form on $L^2(E, \mu)$, 
i.e. we drop condition iv) from the beginning of this section.

Before we present the announced extension, in order to fix the notations, let us recall the following localization procedure: 
Let $G$ be a finely open set and consider the bilinear form 
$$
\mathcal{E}^G(u,v) := \mathcal{E}(u,v) \;\; \mbox{for all} \; u,v \in D(\mathcal{E}^G):=\mathcal{F}_G \; .
$$
Then by \cite{BeBo04}, Theorem 7.6.11 (see also \cite{Os13}, Theorem 3.5.7), 
$(\mathcal{E}^G, \mathcal{F}_G)$ is a (quasi-regular) semi-Dirichlet form 
whose associated process is $X^G$ with state space $G \cup \{ \Delta \}$, obtained by killing $X$ upon leaving $G$:
$$
X^G_t := \left \{
	         \begin{array}{ll}
		         X_t  & \mbox{if } 0\leq t < T_{G^c} \\
		         \Delta & \mbox{if } t \geq T_{G^c}
	         \end{array}
         \right. 
$$
The associated semigroup and resolvent are denoted by $(P_t^G)_{t \geq 0}$ and $(U_\alpha^G)_{\alpha > 0}$.

\begin{thm} \label{thm 5.5}
Let $u \in \mathcal{F}$ and assume there exist a nest $(F_n)_{n\geq1}$ and constants $(c_n)_{n\geq 1}$ such that
$$
\mathcal{E}(u,v) \leq c_n \|v\|_\infty \;\; \mbox{for all} \; v\in \mathcal{F}_{b, F_n}.
$$
Then $\widetilde{u}(X)$ is a semimartingale.
\end{thm}

\begin{proof}
Let us fix a quasi-continuous element $0<f_0 \in \mathcal{F}$ 
and a sequence of positive constants $\alpha_k \nearrow_k \infty$. 
By \cite{MaRo92}, Theorem 2.13 we have that 
$\alpha_k \widehat{U}_{\alpha_k} f_0 \mathop{\longrightarrow}\limits_{k}^{\mathcal{E}_1^{1/2}} f_0$, 
hence by \cite{MaOvRo95}, Proposition 2.18, (i), there exists a nest $(F'_n)_{n\geq1}$ s.t. 
(by passing to a subsequence if necessary) $\lim\limits_{k }\alpha_k \widehat{U}_{\alpha_k} f_0 = f_0$ uniformly on each $F'_n$.
Consequently, replacing $F_n$ with $F_n \cap F'_n$, 
we may assume that $(F_n)_{n\geq 1}$ is a nest such that 
$\sup\limits_k \|1_{F_n} \alpha_k \widehat{U}_{\alpha_k} f_0\|_\infty < \infty$.
Also, without loss of generality we suppose that $\mu (F_n) < \infty$ for all $n \geq 1$. 

Now, let us consider the fine interiors $G_n:= \accentset{\circ}{F}_n^f$ 
and the localized semi-Dirichlet forms $(\mathcal{E}^{G_n}, \mathcal{F}_{G_n})$.
As before, the idea is to localize $u$ by setting 
$$
u_n:= \widetilde{u} - R_1^{G_n^c}\widetilde{u} \in \mathcal{F}_{G_n} \;\; \mbox{for all} \; n \geq 1,
$$
so that by setting $c_k^n := c_n + \alpha_k \|u_n\|_{L^1(G_n, \mu)}+\|u\|_{L^1(G_n, \mu)}$, 
for all $v \in \mathcal{F}_{G_n}$
$$
|\mathcal{E}_{\alpha_k+1}^{G_n}(u_n,v)| = |\mathcal{E}_{\alpha_k+1}(u_n,v)| = 
|\mathcal{E}_{1}(u,v) + \alpha_k (u_n, v)_2| \leq c_k^n \|v\|_\infty.
$$

On the one hand, we claim that $(u_n(X_t)1_{[0, T_{G_n^c})}(t))_{t\geq 0}$ 
is a $\mathbb{P}^x$-semimartingale q.e. $x \in G_n$.
To see this, let us introduce for all $\alpha > 0$ and $n \geq 1$
$$
v_k^n := \alpha_k \widehat{U}_{\alpha_k}^{G_n}(f_0|_{G_n}) \in \mathcal{F}_{G_n}
$$
and note that $v_k^n$ is $\mathcal{E}^{G_n}_{\alpha_k}$-co-excessive 
(i.e. $\widehat{P}_s^{G_n,\alpha_k} v_k^n \leq v_k^n$) 
and $v_k^n \mathop{\longrightarrow}\limits_k^{\mathcal{E}^{G_n}_{1}} f_0|_{G_n} > 0$. 
Furthermore, by the way we chose the nest $(F_n)_{n\geq 1}$
$$
d_n:= \sup\limits_k\|v_k^n\|_\infty \leq \sup\limits_k\|1_{F_n}\alpha_k \widehat{U}_{\alpha_k} f_0 \|_\infty < \infty.
$$ 
It follows that for all $r,t > 0$,

\vspace{0.2cm}
\noindent
$\int_{G_n}P_r^{G_n,\alpha_k+1}|P_t^{G_n,\alpha_k+1} u_n - u_n| v_k^n \; d \mu = |\mathcal{E}_{\alpha_k+1}^{G_n}(u_n,\int_0^t \widehat{P}_s^{G_n,\alpha_k+1}(sgn(P_t^{G_n,\alpha_k+1}u_n -u_n)\widehat{P}_r^{G_n,\alpha_k+1} v_k^n \; ds))|$

\vspace{0.2cm}
\qquad \qquad \qquad\qquad\qquad\quad \;\;\;\;\;\; 
$\leq c_k^n \|\int_0^t \widehat{P}_s^{G_n,{\alpha_k+1}}(sgn(P_t^{G_n,{\alpha_k+1}}u_n -u_n)\widehat{P}_r^{G_n,{\alpha_k+1}} v_k^n \; ds)\|_\infty$

\vspace{0.2cm}
\qquad \qquad \qquad\qquad\qquad\quad \;\;\;\;\;\; 
$\leq c_k^n \int_0^t \|\widehat{P}_{s+r}^{G_n,{\alpha_k+1}}v_k^n\|_\infty \; ds \leq c_k^n \int_0^t e^{-(r+s)}\|v_k^n\|_\infty \; ds$

\vspace{0.2cm}
\qquad \qquad \qquad\qquad\qquad\quad \;\;\;\;\;\; 
$\leq c_k^n d_n e^{- r} t$.

\vspace{0.2cm}
Let now $\tau_l := \{ \frac{i}{2^l} : 0 \leq i \leq l 2^l \}$, $l \geq 1$.
As in the proof of Proposition \ref{prop 4.4}, i) $\Rightarrow$ ii), we get

$$
\int_{G_n}\sup\limits_l V_{\tau_l}^{(P_t^{G_n, \alpha_k+1})}(u_n) \; v_k^n \; d\mu < \infty,
$$
hence, by e.g. Proposition \ref{prop 4.2}, the process $(e^{-(\alpha_k+1)t}u_n(X^{G_n}_t))_{t\geq 0}$, 
and more importantly, the process $u_n(X^{G_n})$, are $\mathbb{P}^x$-semimartingales for 
$v_k^n \cdot \mu$-a.e. $x \in G_n$ for all $k > 0$.
This means that $(u_n(X_t)1_{[0, T_{G_n^c})}(t))_{t\geq 0}$ is a $\mathbb{P}^x$-semimartingale q.e. $x\in G_n$.

On the other hand, as in the proof of Theorem \ref{thm 5.3}, ii) $\Rightarrow$ i), 
the process $(R_1^{G_n^c}\widetilde{u}(X_t)1_{[0, T_{G_n^c})})_{t\geq 0}$ is a semimartingale, hence $\widetilde{u}(X_t)1_{[0, T_{G_n^c})}(t) = u_n(X_t)1_{[0, T_{G_n^c})}(t) + R_1^{G_n^c}\widetilde{u}(X_t)1_{[0, T_{G_n^c})}$, $t\geq 0$ is a semimartingale.
By the result of Meyer already used in Theorem \ref{thm 5.3}, 
it is sufficient to show that $\mathop{\lim}\limits_{n} T_{G_n^c} \geq \xi$ a.s.
But this property is true for $(F_n)_{n \geq 1}$ and it is inherited by $(G_n)_{n\geq 1}$ 
because for any $f\in E(\mathcal{U}^1)$, $R_1^{F_n^c} f = R_1^{G_n^c}f$. 

\end{proof}

The case when $u$ is merely locally in $\mathcal{F}$ 
can be treated just like Corollary \ref{coro 5.4}, so we state this observation as a corollary, but we omit the proof.

\begin{coro} \label{coro 5.6}
The statement of Corollary \ref{coro 5.4} remains valid if $(\mathcal{E}, \mathcal{F})$ is a semi-Dirichlet form.
\end{coro}

\noindent
\subsection{Final remarks} 
For practical reasons, it is useful to know whether it is sufficient to check inequalities 
(\ref{ineq}) and (\ref{lineq}) only for $v \in \mathcal{F}_0$, where $\mathcal{F}_0$ 
is a certain proper subspace of $\mathcal{F}$.
We point out below some ideas of choosing $\mathcal{F}_0$.

a) Assume that $(\mathcal{E}, \mathcal{F})$ is a (non-symmetric) Dirichlet form 
and take $\mathcal{F}_0 := D({\sf \widehat{L}}) \cap L^\infty(\mu)$.
If inequality (\ref{ineq}) is verified for all $v \in \mathcal{F}_0$ 
then $\widetilde{u}(X)$ is an $\alpha$-quasimartingale for all $\alpha > 0$. 
This is true by Proposition \ref{prop 4.5}. 


b) Assume that $(\mathcal{E}, \mathcal{F})$ is a semi-Dirichlet form.
We consider the following extension of condition $(\mathcal{L})$ 
from \cite{Fu99}: a subspace $\mathcal{F}_0 \subset \mathcal{F}_b$ 
satisfies {\it condition} $(\mathcal{S})$ if $\mathcal{F}_0$ is $\mathcal{E}_1^{1/2}$-dense in $\mathcal{F}$ 
and there exists a bounded continuous function $\phi : \mathbb{R} \mapsto \mathbb{R}$ such that

 - $ \phi(\mathcal{F}_0) \subset \mathcal{F}_0$;

 - $\phi (t) = t$ if $t \in [-1,1]$; 

 - if $(v_n)_{n \geq 1} \subset \mathcal{F}_0$ is $\mathcal{E}_1^{1/2}$-convergent then $(\phi(v_n))_n$ is $\mathcal{E}_1^{1/2}$-bounded.  
 
As a candidate for a space satisfying condition ($\mathcal{S}$), 
one should have in mind a {\it core} in the sense of \cite{FuOsTa11} 
(for regular Dirichlet forms), or the space of cylindrical functions 
in the infinite dimensional situation (see e.g. \cite{MaRo92}), while $\phi$ could be a smooth unit contraction. 

In the same spirit as \cite{Fu99}, Lemma 6.1, we have the following result.

\begin{lem} 
Let $u \in \mathcal{F}$ and $\mathcal{F}_0$ satisfy condition ($\mathcal{S}$).
If inequality (\ref{ineq}) holds for all $v \in \mathcal{F}_0$ then it holds for all $v \in \mathcal{F}_b$.
\end{lem}

\begin{proof}
Let $v \in \mathcal{F}$ and assume that $\|v\|_\infty \leq 1$.
If $(v_n)_{n\geq 1} \subset \mathcal{F}_0$ is $\mathcal{E}_1^{1/2}$-convergent to $v$, 
then from the boundedness condition on $(\phi(v_n))_n$ and by Banach-Sacks theorem, there exists a subsequence $(\phi(v_{n_k}))_{n_k}$ whose Cesaro means $\frac{i=1}{k} \sum \limits_1^k\phi(v_{n_i})$ converges to $\phi(v)=v$ w.r.t. $\mathcal{E}_1^{1/2}$.
Therefore $|\mathcal{E}(u,v)| = \lim\limits_{k} |\mathcal{E}(u,\frac{1}{k} \sum \limits_{i=1}^k\phi(v_{n_i}))| \leq c \|\phi \|_\infty$. 
\end{proof}

Regarding inequality (\ref{lineq}), we have:

\begin{lem} \label{lem 5.8}
Let $u \in \mathcal{F}$ and $\mathcal{F}_0$ satisfy condition ($\mathcal{S}$) 
such that inequality (\ref{lineq}) holds for $\mathcal{F}_{b, F_n}$ replaced by $\mathcal{F}_{b, F_n} \cap \mathcal{F}_0$.
In addition, assume that $\mathcal{F}_0$ is an algebra and that for each $n \geq 1$ 
there exists $\psi_n \in \mathcal{F}_0 \cap \mathop{\bigcup}\limits_k \mathcal{F}_{b, F_k}$ 
such that $\psi_n = 1$ on $F_n$.
Then inequality (\ref{lineq}) holds for all $n\geq 1$ and $v \in \mathcal{F}_{b, F_n}$, 
with possible different constants $c_n$.
\end{lem}

\begin{proof}
Fix $n \geq 1$, $v \in \mathcal{F}_{F_n}$ s.t. $\|v\|_\infty \leq 1$, 
and let $k(n) \geq 1$ s.t. $\psi_n \in \mathcal{F}_{b,F_{k(n)}}$.
Take $(v_m)_m \subset \mathcal{F}_0$ which is $\mathcal{E}_1^{1/2}$-convergent to $v$.
Then 
$$
\mathcal{E}_1(\phi(v_m) \psi_n, \phi(v_m) \psi_n) \leq \mathcal{E}_1(\phi(v_m),\phi(v_m))^{1/2} \|\psi_n\|_\infty^{1/2} + \mathcal{E}_1(\psi_n,\psi_n)^{1/2} \|\phi(v_m)\|_\infty^{1/2}
$$ 
which means that $(\phi(v_m) \psi_n)_m$ is $\mathcal{E}_1^{1/2}$-bounded 
and employing once again Banach-Sacks theorem just like we did in the proof of the previous lemma, we get

$$|\mathcal{E}(u,v)| \leq c_{k(n)} \|\phi \psi_n\|_\infty,$$
where the right-hand term does not depend on $v$ 
(in fact it is the new constant replacing $c_n$).

\end{proof}

Candidates for $\mathcal{F}_0$ satisfying the assumption of Lemma \ref{lem 5.8} are the {\it special standard cores} in the sense of \cite{FuOsTa11}; see also \cite{Fu99}, page 27.

\vspace{0.3cm}
\noindent
{\bf Aknowledgements.} 
The first named author acknowledges support from the Romanian National Authority for Scientific Research, 
project number PN-III-P4-ID-PCE-2016-0372. 
The second-named author acknowledges support from the Romanian National Authority for Scientific Research, CNCS-UEFISCDI, project number PN-II-RU-TE-2014-4-0007.

\end{document}